\documentclass[11pt]{ip-journal}

\usepackage{palatino, mathpazo} 

\usepackage{amsmath, amsfonts, amssymb, mathrsfs}
\usepackage[mathscr]{eucal}
\usepackage[all]{xy}
\usepackage{hyperref}
\usepackage{alltt}

\usepackage{marginnote}

\newtheorem{theorem}{Theorem}[section]
\newtheorem{lemma}[theorem]{Lemma}
\newtheorem{proposition}[theorem]{Proposition}
\newtheorem{corollary}[theorem]{Corollary}
\theoremstyle{remark}
\newtheorem{remark}[theorem]{Remark}

\newtheorem{definition}[theorem]{Definition}

\newtheorem*{convention}{Convention}
\numberwithin{equation}{section}

\newcommand{\M}{\mathcal{M}}
\newcommand{\T}{\mathscr{F}}
\newcommand{\Mbar}{\overline{M}}
\newcommand{\p}{\partial}
\newcommand{\f}{{\bf f}}
\newcommand{\bt}{{\bf t}}
\newcommand{\virt}{\operatorname{virt}}
\newcommand{\on}{\operatorname}
\newcommand{\Z}{\mathcal{Z}_1}

\newcommand{\vp}{\vartheta}
\newcommand{\vi}{\iota}
\newcommand{\eo}{\Omega}
\newcommand{\s}{x}

\title[$A + B$ theory in conifold transitions]
{Towards $A + B$ theory in \\conifold transitions\\ for Calabi--Yau threefolds}

\author[Y.-P.~Lee]{Yuan-Pin~Lee}
\address{Y.-P.~Lee: Department of Mathematics, University of Utah,
Salt Lake City, Utah 84112-0090, U.S.A.}
\email{yplee@math.utah.edu}

\author[H.-W.~Lin]{Hui-Wen~Lin}
\address{H.-W.~Lin: Department of Mathematics and Taida
Institute for Mathematical Sciences (TIMS), National Taiwan University, Taipei 10617, Taiwan}
\email{linhw@math.ntu.edu.tw}

\author[C.-L.~Wang]{Chin-Lung~Wang}
\address{C.-L.~Wang: Department of Mathematics and Taida
Institute for Mathematical Sciences (TIMS),
National Taiwan University, Taipei 10617, Taiwan}
\email{dragon@math.ntu.edu.tw}

\begin{document}
\maketitle

\begin{abstract}
For projective conifold transitions between Calabi-Yau threefolds $X$ and $Y$, with $X$ close to $Y$ in the moduli, 
we show that the combined information provided by the $A$ model (Gromov--Witten theory in all genera) and $B$ model (variation of Hodge structures) on $X$, linked along the vanishing cycles, determines the corresponding combined information on $Y$. 
Similar result holds in the reverse direction when linked with the exceptional curves.
%
\end{abstract}

\small
\tableofcontents
\normalsize

\setcounter{section}{-1}

\section{Introduction} \label{s:0}

\subsection{Statements of main results} \label{s:0.1}
Let $X$ be a smooth projective 3-fold. 
A (projective) conifold transition $X \nearrow Y$ is a projective degeneration $\pi: \mathfrak{X} \to \Delta$ of $X$ to a singular variety $\bar X = \mathfrak{X}_0$ with a finite number of ordinary double points (abbreviated as ODPs or nodes) $p_1, \ldots, p_k$, 
locally analytically defined by the equation 
$$x_1^2 + x_2^2 + x_3^2 + x_4^2 = 0,$$
followed by a projective small resolution $\psi: Y \to \bar X$. In the process of complex degeneration from $X$ to $\bar{X}$, $k$ vanishing spheres $S_i \cong S^3$ with trivial normal bundle collapse to nodes $p_i$.
In the process of ``K\"ahler degeneration'' from $Y$ to $\bar{X}$, 
the exceptional loci of $\psi$ above each $p_i$ is a smooth rational curve 
$C_i \cong \mathbb{P}^1$ with $N_{C_i/Y} \cong \mathscr{O}_{\mathbb{P}^1}(-1)^{\oplus 2}$. We write $Y \searrow X$ for the reverse process.

Notice that $\psi$ is a crepant resolution and $\pi$ is a finite distance degeneration with respect to the quasi-Hodge metric \cite{cW0, cW1}. A transition of this type (in all dimensions) is called an extremal transition. In contrast to the usual \emph{birational} $K$-equivalence, an extremal transition may be considered as a \emph{generalized} $K$-equivalence in the sense that the small resolution $\psi$ is crepant and the degeneration $\pi$ preserves sections of the canonical bundle. It is generally expected that simply connected Calabi--Yau 3-folds are connected through extremal transitions, of which conifold transitions are the most fundamental. (This has been extensively checked numerically \cite{cyweb}.)  
It is therefore a natural starting point of investigation.

We study the changes of the so-called $A$ model and $B$ model under a projective conifold transition. In this paper, the $A$ model is the Gromov--Witten (GW) theory of all genera; the $B$ model is the variation of Hodge structures (VHS),
which is in a sense only the genus zero part of the quantum $B$ model. 

In general, the conditions for the existence of projective conifold transitions is an unsolved problem except in the case of Calabi--Yau 3-folds,
for which we have fairly good understanding.
For the inverse conifold transition $Y \searrow X$, a celebrated theorem 
of Friedman \cite{rF} (see also \cite{yK, gT}) states that 
a small contraction $Y \to \bar{X}$ can be smoothed if and only if 
there is a totally nontrivial relation between the exceptional curves. That is, there exist constants $a_i \neq 0$ for all $i=1, \ldots, k$ such that $\sum_{i=1}^k a_i [C_i] =0$.
These are relations among curves $[C_i]$'s in the kernel of $H_2(Y)_{\mathbb{Z}} \to H_2(X)_{\mathbb{Z}}$.
Let $\mu$ be the number of independent relations and let $A \in M_{k \times \mu} (\mathbb{Z})$ be a relation matrix for $C_i$'s,
in the sense that the column vectors span all relations.
Conversely, for a conifold transition $X \nearrow Y$, Smith, Thomas and Yau proved a dual statement in \cite{STY}, asserting that 
the $k$ vanishing 3-spheres $S_i$ must satisfy a totally nontrivial relation $\sum_{i=1}^k b_i [S_i] =0$ 
 in $V_{\mathbb{Z}}:= \ker (H_3(X)_{\mathbb{Z}} \to H_3(\bar{X})_{\mathbb{Z}})$ with $b_i \neq 0$ for all $i$.
Let $\rho$ be the number of independent relations and $B \in M_{k \times \rho} (\mathbb{Z})$ be a relation matrix for $S_i$'s.
It turns out that $\mu + \rho =k$ \cite{hC} and the following exact sequence holds.

\begin{theorem}[= Theorem~\ref{t:bes}] \label{t:0.1}
Under a conifold transition $X \nearrow Y$ of smooth projective threefolds, we have an exact sequence of weight two Hodge structures:
\begin{equation} \label{e:0.1}
0 \to H^2(Y)/H^2(X) \stackrel{B}{\longrightarrow} \mathbb{C}^k \stackrel{A^t}{\longrightarrow} V  \to 0.
\end{equation}
\end{theorem}

We interpret this as a partial exchange of topological information between the \emph{excess $A$ model} of $Y/X$ (in terms of $H^2(Y)/H^2(X)$)
and the \emph{excess $B$ model} of $X/Y$ in terms of the space of vanishing cycles $V$. 

To study the changes of quantum $A$ and $B$ models under a projective conifold transition of Calabi--Yau 3-folds and its inverse,
the first step is to find a $\mathcal{D}$-module version of Theorem~\ref{t:0.1}.
We state the result below in a suggestive form and leave the precise statement to Theorem~\ref{p:qbes}:

\begin{theorem}[= Theorem~\ref{p:qbes}] \label{t:0.1.5}
Via the exact sequence \eqref{e:0.1}, the trivial logarithmic connection on $(\underline{\mathbb{C}} \oplus \underline{\mathbb{C}}^{\vee})^k \to \mathbb{C}^k$ induces simultaneously the logarithmic part of the Gauss--Manin connection on $V$ and the Dubrovin connection on $H^2(Y)/H^2(X)$.
\end{theorem}

Note that the Gauss--Manin connection on $V$ determines the excess $B$ model and Dubrovin connection on $H^2(Y)/H^2(X)$ determines the excess $A$ model in genus zero.
The logarithmic part of the connection determines the residue connection and hence the monodromy. One can interpret Theorem~\ref{t:0.1.5} heuristically as "excess $A$ theory $+$ excess $B$ theory $\sim$ trivial''. In other words, the logarithmic parts of two flat connections on excess theories ``glues'' to form a trivial theory. This gives a strong indication towards a unified $A + B$ theory.

``Globalizing'' this result, i.e., going beyond the excess theories, is the next step towards a true $A + B$ theory, which is still beyond immediate reach. Instead we will settle for results on mutual determination in implicit form. Recall that the Kuranishi spaces $\M_X$, $\M_Y$ of Calabi--Yau manifolds are unobstructed (the Bogomolov--Tian--Todorov theorem). For a Calabi--Yau conifold $\bar X$, the unobstructedness of $\M_{\bar X}$  also holds \cite{yK, gT, N}.

\begin{theorem} \label{t:0.2}
Let $X \nearrow Y$ be a projective conifold transition of Calabi--Yau threefolds such that $[X]$ is a nearby point of $[\bar{X}]$ in $\M_{\bar{X}}$. Then 
\begin{enumerate}
\item[(1)] $A(X)$ is a sub-theory of $A(Y)$. 
\item[(2)] $B(Y)$ is a sub-theory of $B(X)$.
\item[(3)] $A(Y)$ can be reconstructed from a refined $A$ model of $X^{\circ} := X \setminus \bigcup_{i=1}^k S_i$ ``linked'' by the vanishing spheres in $B(X)$.
\item[(4)] $B(X)$ can be reconstructed from a refined $B$ model of $Y^{\circ} := Y \setminus \bigcup_{i=1}^k C_i$ ``linked'' by the exceptional curves in $A(Y)$.
\end{enumerate}
\end{theorem}

The meaning of these slightly obscure statements will take the entire paper to spell them out. It may be considered as a \emph{categorification} of Clemens' identity $\mu + \rho = k$. Here we give only brief explanations.

(1) is mostly due to Li--Ruan, who in \cite{LR} pioneered the mathematical study of conifold transitions in GW theory.
The proof follows from degeneration arguments and existence of flops (cf.~Proposition~\ref{p:1}).

For (2), we note that there are natural identifications of $\M_Y$ with
the boundary of $\M_{\bar{X}}$ consisting of equisingular deformations,
and $\M_X$ with $\M_{\bar{X}} \setminus \mathfrak{D}$ where the discriminant locus $\mathfrak{D}$ is a \emph{central hyperplane arrangement} with axis $\M_Y$ (cf.~\S\ref{s:4.4}).
Therefore, the VHS associated to $Y$ can be considered as a sub-VHS system
of VMHS associated to $\bar{X}$ (cf.~Corollary \ref{c:sub-sys}), 
which is a regular singular extension of the VHS associated to $X$.

With (3), we introduce the ``linking data'' of the holomorphic
curves in $X^{\circ}$, which not only records the curve classes in $X$ but also
how the curve links with the vanishing spheres $\bigcup_i S_i$.
The linking data on $X$ can be identified with the curve classes in $Y$ by 
$H_2(X^{\circ}) \cong H_2(Y)$ (cf.~Definition~\ref{d:2} and \eqref{e:linking}).
We then proceed to show, by the degeneration argument, that 
the virtual class of moduli spaces of stable maps to $X^\circ$ is naturally a disjoint
union of pieces labeled by elements of the linking data (cf.~Proposition \ref{p:4}).
Furthermore, the Gromov--Witten invariants 
in $Y$ is the same as the numbers produced by the component of the virtual class on $X$ labeled by the corresponding linking data.
Thus, the refined $A$ model is really the ``linked $A$ model'' and is equivalent to the (usual) $A$ model of $Y$ (for non-extremal curves classes) in all genera.
The vanishing cycles from $B(X)$ plays a key role in reconstructing $A(Y)$.

For (4), the goal is to reconstruct VHS on $\M_X$ from VHS on $\M_Y$ and $A(Y)$.
The deformation of $\bar{X}$ is unobstructed. Moreover it is well known that
$\on{Def}(\bar{X}) \cong H^1 (Y^{\circ}, T_{Y^{\circ}})$.
Even though the deformation of $Y^{\circ}$ is obstructed (in the direction transversal to $\M_Y$),
there is a first order deformation parameterized by $H^1 (Y^{\circ}, T_{Y^{\circ}})$
which gives enough initial condition to uniquely determine the degeneration of
Hodge bundles on $\M_{\bar{X}}$ near $\M_Y$. A technical result needed in this process is a short exact sequence 
$$0 \to V \to H^3(X) \to H^3(Y^\circ) \to 0$$
which connects the limiting mixed Hodge structure (MHS) of Schmid on $H^3(X)$ and the canonical MHS of Deligne on $H^3(Y^\circ)$ (cf.~Proposition \ref{p:lift}). 
Together with the monodromy data associated to the ODPs, which is encoded in the relation matrix $A$ of the extremal rays on $Y$, we will be able to determine the VHS on $\M_X$ near $\M_Y$. 
In the process, an extension of Schmid's nilpotent orbit theorem \cite{Schmid} to degenerations with certain non-normal crossing discriminant loci is also needed.
See Theorem~\ref{p:gnot} for details.

\subsection{Motivation and future plans}

Our work is inspired by the famous Reid's fantasy \cite{mR}, where conifold transitions play a key role in
connecting irreducible components of moduli of Calabi--Yau threefolds.
Theorems~\ref{t:0.1.5} and \ref{t:0.2} above can be interpreted as the partial exchange
of $A$ and $B$ models under a conifold transition.
We hope to answer the following intriguing question concerning with ``global symmetries'' on moduli spaces of Calabi--Yau 3-folds in the future:
\emph{Would this partial exchange of $A$ and $B$ models lead to 
``full exchange'' when one connects a Calabi--Yau threefold to its mirror via a finite steps of extremal transitions? If so, what is the relation between this full exchange and the one induced by ``mirror symmetry''?} To this end, we need to devise a computationally effective way to achieve 
explicit determination of this partial exchange.
One missing piece of ingredients in this direction is a blowup formula 
in the Gromov--Witten theory for conifolds, 
which we are working on and have had some partial success \cite{LLW3}. (For smooth blowups with complete intersection centers, we have a fairly good solution in genus zero.)

More speculatively, the mutual determination of $A$ and $B$ models 
on $X$ and $Y$ leads us to surmise the possibility of a unified 
``$A+B$ model'' which will be \emph{invariant} under any extremal transition.
For example, the string theory predicts that Calabi--Yau threefolds form 
an important ingredient of our universe,
but it does not specify which Calabi--Yau threefold we live in.
Should the $A+B$ model be available and proven invariant under extremal transitions,
one would then have no need to make such a choice.

The first step of achieving this goal is to generalize Theorem~\ref{t:0.1.5} to the full local theory, including the non-log part of the connections.
We note that the excess $A$ model on $H^2(Y/X)$ can be extended to the (flat) Dubrovin connection on $Y$
while the excess $B$ model on $H^3(X/Y)$ can be extended to the (flat) Gauss--Manin connection on $X$. We hope to be able to ``glue'' the complete $A$ model on $Y$ and the complete $B$ model on $X$ as flat connections on the unified K\"ahler plus complex moduli.

\subsection{Acknowledgements}

We are grateful to C.H.\ Clemens, C.-C.\ M.\ Liu, M.\ Rossi, I.\ Smith and R.\ Thomas for discussions related to this project, and to the referee for his/her valuable suggestions. 

Y.-P.~Lee's research is partially funded by the National Science Foundation. H.-W.~Lin and C.-L. Wang are both supported by the Ministry of Science and Technology, Taiwan. 
We are also grateful to Taida Institute of Mathematical Sciences (TIMS) for its generous and constant support which makes this long term collaboration possible.

The vast literature on extremal transitions makes it impossible to cite all related articles in this area, and
we apologize for omission of many important papers which are not directly related to our approach.

\section{The basic exact sequence from Hodge theory} \label{s:2}
In this section, we recall some standard results on the geometry of projective conifold transitions.
Definitions and short proofs are mostly spelled out to fix the notations, even when they are well known.
Combined with well-known tools in Hodge theory, we derive the \emph{basic exact sequence}, 
which is surprisingly absent in the vast literature on the conifold transitions.

\begin{convention}
In \S 1-2, all discussions are for projective conifold transitions 
\emph{without the Calabi--Yau condition}, unless otherwise specified.
The Calabi--Yau condition is imposed in \S 3-5. 
Unless otherwise specified, cohomology groups are over $\mathbb{Q}$ when only topological aspect (including weight filtration) is concerned;
they are considered over $\mathbb{C}$ when the (mixed) Hodge-theoretic aspect is involved. 
All equalities, whenever make sense in the context of mixed Hodge structure (MHS), hold as equalities for MHS.
\end{convention}

\subsection{Preliminaries on conifold transitions} 

The results here are mostly contained in \cite{hC} and are included here for readers' convenience.

\subsubsection{Local geometry} \label{s:1.2}

Let $X$ be a smooth projective 3-fold and $X \nearrow Y$ a \emph{projective conifold transition} through $\bar X$ with nodes $p_1, \ldots, p_k$ as in \S \ref{s:0.1}. Locally analytically, a node (ODP) is defined by the equation
\begin{equation} \label{e:A1}
  x_1^2 + x_2^2 + x_3^2 + x_4^2 = 0,
\end{equation}
or equivalently $uv - ws =0$. The small resolution $\psi$ can be achieved by blowing up the Weil divisor
defined by $u=w=0$ or by $u = s = 0$, these two choices differ by a flop.

\begin{lemma} \label{l:1}
The exceptional locus of $\psi$ above each $p_i$ is a smooth rational curve 
$C_i$ with $N_{C_i/Y} \cong \mathscr{O}_{\mathbb{P}^1}(-1)^{\oplus 2}$. Topologically, $N_{C_i/Y}$ is a trivial rank $4$ real bundle.
\end{lemma}

\begin{proof}
Away from the isolated singular points $p_i$'s, the Weil divisors
are Cartier and the blowups do nothing.
Locally near $p_i$, the Weil divisor is generated by two functions $u$ and $w$.
The blowup $Y \subset \mathbb{A}^4 \times \mathbb{P}^1$ is defined by $z_0 v - z_1 s =0$,
in addition to $uv-ws=0$ defining $X$, 
where $(z_0: z_1)$ are the coordinates of $\mathbb{P}^1$. Namely we have ${u}/{w} = {s}/{v} = {z_0}/{z_1}$. It is now easy to see the exceptional locus near $p_i$ is isomorphic to $\mathbb{P}^1$
and the normal bundle is as described (by the definition of $\mathscr{O}_{\Bbb P^1}(-1)$). Since oriented $\Bbb R^4$-bundles on $\Bbb P^1 \cong S^2$ are classified by the second Stiefel--Whitney class $w_2$ (via $\pi_1({\rm SO}(4)) \cong \Bbb Z/2$), the last assertion follows immediately.
\end{proof}

Locally to each node $p = p_i \in \bar X$, 
the transition $X \nearrow Y$ can be considered as two different ways 
of ``smoothing'' the singularities in $\bar{X}$:
deformation leads to $X_t$ and small resolution leads to $Y$.
Topologically, we have seen that the exceptional loci of $\psi$ are 
$\coprod_{i=1}^k C_i$, a disjoint union of $k$ 2-spheres.
For the deformation, the classical results of Picard, Lefschetz and Milnor
state that there are $k$ vanishing 3-spheres $S_i \cong S^3$.

\begin{lemma} \label{l:2}
The normal bundle $N_{S_i/X_t} \cong T^*_{S_i}$ is a trivial rank $3$ real bundle.
\end{lemma}

\begin{proof}
From \eqref{e:A1}, after a degree two base change the local equation of the family near
an ODP is 
$$\sum\nolimits_{j = 1}^4 x_j^2 = t^2 = |t|^2 e^{2\sqrt{-1}\theta}. $$
Let $y_j = e^{\sqrt{-1} \theta} x_j$ for $j = 1, \ldots, 4$, the equation leads to
\begin{equation} \label{e:1}
 \sum\nolimits_{j = 1}^4 y_j^2 = |t|^2.
\end{equation}
Write $y_j$ in terms of real coordinates $y_j = a_j + \sqrt{-1} b_j$,
we have $|\vec{a}|^2 = |t|^2 + |\vec{b}|^2$ \text{and} $\vec{a} \cdot \vec{b} =0$,
where $\vec{a}$ and $\vec{b}$ are two vectors in $\mathbb{R}^4$.
The set of solutions can be identified with 
$T^* S_r$ with the bundle structure $T^* S_r \to S_r$ defined by $(\vec{a}, \vec{b}) \mapsto  r {\vec{a}}/{|\vec{a}|} \in S_r$
where $S_r$ is the 3-sphere with radius $r =|t|$. The vanishing sphere can be chosen to be the real locus of the equation of \eqref{e:1}. Therefore, $N_{S_r/X_t}$ is naturally identified with the cotangent bundle $T^*{S_r}$, which is a trivial bundle since $S^3 \cong SU(2)$ is a Lie group. 
\end{proof}

\begin{remark} \label{r:1}
The vanishing spheres above are Lagrangian with respect to the natural symplectic structure on $T^* S^3$. A theorem of Seidel and Donaldson \cite{skD} states that this is true globally, namely the vanishing spheres can be chosen to be Lagrangian with respect to the symplectic structure coming from the K\"ahler structure of $X_t$.
\end{remark}

By Lemma~\ref{l:2}, the $\delta$ neighborhood of the vanishing 3-sphere 
$S_r^3$ in $X_t$ is diffeomorphic to the trivial disc bundle $S^3_r \times D^3_\delta$.

By Lemma~\ref{l:1} the $r$ neighborhood of the exceptional 2-sphere $C_i= S^2_{\delta}$ 
is $D^4_r \times S^2_{\delta}$, where $\delta$ is the radius defined by 
$4\pi \delta^2 = \int_{C_i} \omega$ for the background K\"ahler metric $\omega$.

\begin{corollary} \cite[Lemma 1.11]{hC} \label{c:surgery}
On the topological level one can go between $Y$ and $X_t$ by surgery via
$$\p (S^3_r \times D_\delta^3) = S^3_r \times S^2_\delta = \p(D^4_r \times S^2_\delta).$$
\end{corollary}

\begin{remark}[Orientations on $S^3$] \label{orient}
The two choices of orientations on $S^3_r$ induces two different surgeries. The resulting manifolds $Y$ and $Y'$ are in general not even homotopically equivalent. In the complex analytic setting the induced map $Y \dasharrow Y'$ is known as an ordinary (Atiyah) flop. 
\end{remark}

\subsubsection{Global topology}

\begin{lemma} \label{l:3}
Define 
$$\mu := \tfrac{1}{2} (h^3(X) - h^3(Y)) \quad \text{and} \quad \rho := h^2(Y) - h^2(X).$$ 
Then,
\begin{equation} \label{e:3}
  \mu + \rho = k.
\end{equation}
\end{lemma}

\begin{proof}
The Euler numbers satisfy 
$$\chi(X) - k\chi(S^3) = \chi(Y) - k \chi(S^2).$$
That is, 
$$2 - 2h^1(X) + 2h^2(X) - h^3(X) = 2 - 2h^1(Y) + 2h^2(Y) - h^3(Y) - 2k.$$ 
By the above surgery argument we know that 
conifold transitions preserve $\pi_1$.
Therefore, $\tfrac{1}{2} (h^3(X) - h^3(Y)) + (h^2(Y) - h^2(X)) = k$.
\end{proof}

\begin{remark}
In the Calabi-Yau case, $\mu = h^{2, 1}(X) - h^{2, 1}(Y) = -\Delta h^{2, 1}$ is the lose of complex moduli, and $\rho = h^{1, 1}(Y) - h^{1, 1}(X) = \Delta h^{1, 1}$ is the gain of K\"ahler moduli. Thus \eqref{e:3} is really 
$$
\Delta (h^{1, 1} - h^{2, 1}) = k = \tfrac{1}{2}\Delta \chi.
$$ 
\end{remark}

In the following, we study the \emph{Hodge-theoretic meaning} of
\eqref{e:3}.

\subsection{Two semistable degenerations} \label{s:2.1}
To apply Hodge-theoretic methods on degenerations,
we factor the transition $X \nearrow Y$ as a composition of two 
semistable degenerations 
$\mathcal{X} \to \Delta$ and $\mathcal{Y} \to \Delta$. 

The \emph{complex degeneration} 
$$
f: \mathcal{X} \to \Delta
$$
is the semistable reduction of $\mathfrak{X} \to \Delta$ obtained by a 
degree two base change $\mathfrak{X}' \to \Delta$ followed by the blow-up 
$\mathcal{X} \to \mathfrak{X}'$ of all the four dimensional nodes 
$p_i' \in \mathfrak{X}'$. 
The special fiber 
$\mathcal{X}_0 = \bigcup_{j = 0}^k X_j$ 
is a simple normal crossing divisor with 
$$\tilde\psi: X_0 \cong \tilde{Y}:= \on{Bl}_{\coprod_{i=1}^k \{p_i\}} \bar{X} \to \bar X$$ 
being the blow-up at the nodes and with 
$$X_i = Q_i \cong Q \subset \mathbb{P}^4, \quad i = 1, \ldots, k$$
being quadric threefolds. 
Let $X^{[j]}$ be the disjoint union of $j + 1$ intersections from $X_i$'s. 
Then the only nontrivial terms are $X^{[0]} = \tilde Y \coprod_i Q_i$ and $X^{[1]} = \coprod_i E_i$ where $E_i = \tilde Y \cap Q_i \cong \mathbb{P}^1 \times \mathbb{P}^1$ are the $\tilde\psi$ exceptional divisors. The semistable reduction $f$ does not require the existence of a small resolution of $\mathfrak{X}_0$. 

The \emph{K\"ahler degeneration} 
$$
g: \mathcal{Y} \to \Delta
$$
is simply the deformations to the normal cone $\mathcal{Y} = {\rm Bl}_{\coprod C_i \times \{0\}} Y \times \Delta \to \Delta$. The special fiber 
$\mathcal{Y}_0 = \bigcup_{j = 0}^k Y_j$ with 
$$\phi: Y_0 \cong \tilde Y := {\rm Bl}_{\coprod_{i = 1}^k \{C_i\}}\,Y \to Y$$ 
being the blow-up along the curves $C_i$'s and 
$$Y_i = \tilde E_i \cong \tilde E := {P}_{\mathbb{P}^1} (\mathscr{O}(-1)^2 \oplus \mathscr{O}), \quad i = 1, \ldots, k.$$ 
In this case the only non-trivial terms for $Y^{[j]}$ are 
$Y^{[0]} = \tilde Y \coprod_i \tilde E_i$ and $Y^{[1]} = \coprod_i E_i$ where $E_i = \tilde Y \cap \tilde E_i$ is now understood as the infinity divisor (or relative hyperplane section) 
of $\pi_i: \tilde E_i \to C_i \cong \mathbb{P}^1$.

\subsection{Mixed Hodge Structure and the Clemens--Schmid exact sequence} \label{s:2.2}

We now apply the Clemens--Schmid exact sequence \cite{hC2} to the above two semistable degenerations. A general reference is \cite{pG1}. We will mainly be interested in $H^{\leq 3}$. The computation of $H^{> 3}$ is similar.

\subsubsection{}
The cohomology of $H^*(\mathcal{X}_0)$, with its canonical mixed Hodge structure, 
is computed from the spectral sequence $E_0^{p, q}(\mathcal{X}_0) = \Omega^q(X^{[p]})$ 
with $d_0 = d$, the de Rham differential, and then 
$$E_1^{p, q}(\mathcal{X}_0) = H^q(X^{[p]})$$ with $d_1 =\delta$ being the combinatorial coboundary operator 
$$\delta: H^q(X^{[p]}) \to H^q(X^{[p + 1]}).$$ 
The spectral sequence degenerates at $E_2$ terms. 

The weight filtration on $H^*(\mathcal{X}_0)$ is induced from the increasing filtration on the spectral sequence
$W_m := \bigoplus_{q \leq m} E^{*, q}.$
Therefore, 
$$\on{Gr}^W_m(H^j) = E_2^{j-m, m}, \quad \on{Gr}^W_m (H^j) = 0 \quad \text{for $m < 0$ or $m > j$}.$$ 
Since $X^{[j]} \neq \emptyset$ only when $j = 0, 1$, we have
\[
 H^0 \cong E_2^{0,0}, \quad H^1 \cong E_2^{1,0} \oplus E_2^{0,1}, \quad
 H^2 \cong E_2^{1,1} \oplus E_2^{0,2}, \quad H^3 \cong E_2^{1,2} \oplus E_2^{0,3}.
\]
The only weight $3$ piece is $E_2^{0,3}$, which can be computed by 
$$
 \delta: E_1^{0,3} = H^3(X^{[0]}) \mathop{\longrightarrow} E_1^{1,3}=H^3(X^{[1]}).
$$
Since $Q_i$, $\tilde E_i$ and $E_i$ have no odd cohomologies, 
$H^3(X^{[1]}) =0$ and $H^3(X^{[1]}) = H^3(\tilde Y)$.
We have thus $E_2^{0,3} = H^3(\tilde Y)$.

The weight 2 pieces, which is the most essential part, is computed from 
\begin{equation} \label{e:4}
H^2(X^{[0]}) = H^2(\tilde Y) \oplus \bigoplus\nolimits_{i = 1}^k H^2(Q_i) \mathop{\longrightarrow}\limits^{\delta_2} H^2(X^{[1]}) = \bigoplus\nolimits_{i = 1}^k H^2(E_i).
\end{equation}
We have $E_2^{1,2} = \on{cok}(\delta_2)$ and $E_2^{0,2} =\ker (\delta_2)$.
The weight 1 and weight 0 pieces can be similarly computed. For weight 1 pieces we have 
$$E_2^{0,1} = H^1(X^{[0]}) = H^1(\tilde Y) \cong H^1(Y) \cong H^1(X),$$ 
and $E_2^{1,1} = 0$.
The weight 0 pieces are computed from $\delta: H^0(X^{[0]}) \to H^0(X^{[1]})$ and we have $E_2^{0,0} = H^0(\tilde Y) \cong H^0(Y) \cong H^0(X)$, and $E_2^{1,0} = 0$. We summarize these calculations as

\begin{lemma} \label{l:4}
There are isomorphisms of MHS:
\[
\begin{split}
 &H^3(\mathcal{X}_0) \cong H^3(\tilde Y) \oplus \on{cok}(\delta_2), \\
 &H^2(\mathcal{X}_0) \cong \ker(\delta_2),\\
 &H^1(\mathcal{X}_0) \cong H^1(\tilde Y) \cong H^1(Y) \cong H^1(X),\\
 &H^0(\mathcal{X}_0) \cong H^0(\tilde Y) \cong H^0(Y) \cong H^0(X).
\end{split}
\]
In particular, $H^j(\mathcal{X}_0)$ is pure of weight $j$ for $j \le 2$.
\end{lemma}

\subsubsection{} \label{s:2.2.2}
Here we give a dual formulation of \eqref{e:4} which will be useful later.
Let $\ell, \ell'$ be the line classes of the two rulings of
$E \cong \mathbb{P}^1 \times \mathbb{P}^1$. 
Then $H^2(Q, \mathbb{Z})$ is generated by $e = [E]$ as a hyperplane class and 
$e|_E = \ell + \ell'$. 
The map $\delta_2$ in \eqref{e:4} is then equivalent to
\begin{equation} \label{e:delta2bar}
\bar{\delta}_2: H^2(\tilde Y) \longrightarrow \bigoplus\nolimits_{i = 1}^k H^2(E_i)/H^2(Q_i).
\end{equation}
Since $H^2(\tilde Y) = \phi^* H^2(Y) \oplus \bigoplus_{i = 1}^k \langle [E_i] \rangle$ and
$[E_i]|_{E_i} = -(\ell_i + \ell_i')$, the second component $\bigoplus_{i = 1}^k \langle [E_i] \rangle$ lies in $\ker (\bar{\delta}_2)$ and $\bar{\delta}_2$ factors through 
\begin{equation} \label{e:5}
\phi^* H^2(Y) \to \bigoplus\nolimits_{i = 1}^k H^2(E_i)/H^2(Q_i) \cong \bigoplus\nolimits_{i = 1}^k \langle \ell_i - \ell_i'\rangle
\end{equation}
(as $\Bbb Q$-spaces). Notice that the quotient is isomorphic to $\bigoplus_{i = 1}^k \langle \ell_i'\rangle$ integrally. 

By reordering we may assume that $\phi_* \ell_i = [C_i]$ and $\phi^* [C_i] = \ell_i - \ell_i'$ (cf.~\cite{LLW1}). The dual of \eqref{e:5} then coincides with
the fundamental class map 
$$\vp: \bigoplus\nolimits_{i = 1}^k \langle [C_i] \rangle \longrightarrow H_2(Y).$$ 
In general for a $\mathbb{Q}$-linear map $\vp: P \to Z$, we have $\on{im} \vp^* \cong (P/\ker \vp)^* \cong (\on{im} \vp)^*$. Thus 
\begin{equation} \label{e:6}
 \dim_{\mathbb{Q}} \on{cok} (\delta_2)  + \dim_{\mathbb{Q}} \on{im} (\vp) = k.
\end{equation} 

We will see in Corollary~\ref{c:1} that $\dim \on{cok} \delta = \mu$ and 
$\dim \on{im} \vp = \rho$. This gives the Hodge theoretic meaning of $\mu + \rho = k$ in Lemma~\ref{l:3}. Further elaboration of this theme will follow in Theorem~\ref{t:bes}.

\subsubsection{}
On $\mathcal{Y}_0$, the computation is similar and a lot easier. 
The weight 3 piece can be computed by the map $H^3(Y^{[0]}) = H^3(\tilde Y) \longrightarrow H^3(Y^{[1]}) = 0$; the weight 2 piece is similarly computed by the map
$$
H^2(Y^{[0]}) = H^2(\tilde Y) \oplus \bigoplus\nolimits_{i = 1}^k H^2(\tilde E_i) \mathop{\longrightarrow}\limits^{\delta'_2} H^2(Y^{[1]}) = \bigoplus\nolimits_{i = 1}^k H^2(E_i).
$$
Let $h = \pi^* ({\rm pt})$ and $\xi = [E]$ for $\pi: \tilde E \to \mathbb{P}^1$. Then $h|_E = \ell'$ and $\xi|_E = \ell + \ell'$. 
In particular the restriction map $H^2(\tilde E) \to H^2(E)$ is an isomorphism and 
hence $\delta'_2$ is surjective.
The computation of pieces from weights 1 and 0 is the same as for $\mathcal{X}_0$.
We have therefore the following lemma.

\begin{lemma} \label{l:5}
There are isomorphisms of MHS:
\[
\begin{split}
 &H^3(\mathcal{Y}_0) \cong H^3(Y^{[0]}) \cong H^3(\tilde Y) , \\
 &H^2(\mathcal{Y}_0) \cong \ker(\delta'_2) \cong H^2(\tilde Y),\\
 &H^1(\mathcal{Y}_0) \cong H^1(\tilde Y) \cong H^1(Y) \cong H^1(X),\\
 &H^0(\mathcal{Y}_0) \cong H^0(\tilde Y) \cong H^0(Y) \cong H^0(X).
\end{split}
\]
\end{lemma}

\subsubsection{}
We denote by $N$ the monodromy operator for both $\mathcal{X}$ and $\mathcal{Y}$ families. The map $N$ induces the unique monodromy weight filtrations $W$ on $H^n(X)$ which, together with the limiting Hodge filtration $F_\infty^\bullet$, leads to Schmid\rq{}s limiting MHS \cite{Schmid, jS}. That is, 
$$0 \subset W_0 \subset W_1 \subset \cdots \subset W_{2n - 1} \subset  W_{2n} = H^n(X)$$ 
such that $N W_k \subset W_{k - 2}$ and for $\ell \ge 0$,
\begin{equation} \label{e:H-sym}
N^\ell: G^W_{n + \ell} \cong G^W_{n - \ell}
\end{equation}
on graded pieces. The induced filtration $F^p_\infty G^W_k := F^p_\infty \cap W_k/F^p_\infty \cap W_{k - 1}$ defines a pure Hodge structure of weight $k$ on $G^W_k$. Similar constructions apply to $H^n(Y)$ as well.

\begin{lemma} \label{l:6}
We have the following exact sequences (of MHS) for $H^2$ and $H^3$:
\begin{equation*} 
\begin{split}
 0 \to H^3(\mathcal{X}_0) \to & H^3(X) \mathop{\longrightarrow}\limits^{N} H^3(X) \to H_3(\mathcal{X}_0) \to 0,\\
 0 \to H^0(X) \to H_6(\mathcal{X}_0) \to H^2(\mathcal{X}_0) \to & H^2(X) \mathop{\longrightarrow}\limits^N 0, \\
 0 \to H^3(\mathcal{Y}_0) \to & H^3(Y) \mathop{\longrightarrow}\limits^N 0,\\
 0 \to H^0(Y) \to H_6(\mathcal{Y}_0) \to H^2(\mathcal{Y}_0) \to & H^2(Y) \mathop{\longrightarrow}\limits^N 0.
\end{split}
\end{equation*}
\end{lemma}

\begin{proof}
These follow from the Clemens--Schmid exact sequence, which is compatible with the MHS. The other terms in the first sequence, namely $H^1(X)\to H_5(\mathcal{X}_0)$ to the left end and $H^5(\mathcal{X}_0) \to H^5(X)$ to the right end, can be ignored since they induce isomorphisms, as can be checked using MHS on $H_5(\mathcal{X}_0)$. 
Similar comments apply to the third sequence for $H^3(Y)$. 
 
Note that the monodromy is trivial for $\mathcal{Y} \to \Delta$ since the punctured family is trivial. For the second sequence, by Lemma~\ref{l:4}, we know that $H^2(\mathcal{X}_0)$ is pure of weight 2. Hence $N$ on $H^2(X)$ is also trivial and the Hodge structure does not degenerate. Indeed, if $N \ne 0$ then $\ker N$ contains some part of weight $\le 2$ by \eqref{e:H-sym}.
\end{proof}

\begin{corollary} \label{c:1}
\begin{itemize}
\item[(i)] $\rho = \dim \on{im}(\vp)$ and $\mu = \dim \on{cok}(\delta_2)$.

\item[(ii)] 
$H^3(Y) \cong H^3(\mathcal{Y}_0) \cong H^3(Y^{[0]}) \cong H^3(\tilde Y) \cong \on{Gr}^W_3 H^3(X)$.

\item[(iii)] Denote by $K := \ker (N: H^3(X) \to H^3(X))$. Then $H^3(\mathcal{X}_0) \cong K$. More precisely, $\on{Gr}^W_3 (H^3(\mathcal{X}_0))  \cong H^3(Y)$ and $\on{Gr}^W_2 (H^3(\mathcal{X}_0))  \cong \on{cok} (\delta_2)$.
\end{itemize}
\end{corollary}

\begin{proof}
By Lemma~\ref{l:4}, $h^2(\mathcal{X}_0) = \dim \ker (\delta_2)$.
It follows from the second and the fourth exact sequences in Lemma~\ref{l:6} that $h^2(X) = \dim \ker (\delta_2) + 1 - (k + 1)$. Rewrite \eqref{e:4} as
\begin{equation} \label{e:7}
0 \to \ker (\delta_2) \to H^2(X^{[0]}) \mathop{\longrightarrow}\limits^\delta 
 H^2(X^{[1]}) \to \on{cok} (\delta_2) \to 0,
\end{equation}
which implies $\dim \ker(\delta_2) + 2k = \dim \on{cok} (\delta_2) + 2 k + h^2(Y)$. 

Combining these two equations with \eqref{e:6}, we have $\rho = h^2(Y) - h^2(X) = k - \dim \on{cok} (\delta_2) = \dim \on{im}(\vp)$. This proves the first equation for $\rho$ in (i).

Combining the first equation in Lemma~\ref{l:5} and the third exact sequence in Lemma~\ref{l:6}, we have
\begin{equation}  \label{e:8}
H^3(Y) \cong H^3(\mathcal{Y}_0) \cong H^3(\tilde Y) .
\end{equation}
This shows (ii) except the last equality.

By Lemmas~\ref{l:6} and \ref{l:4}, $K \cong H^3(\mathcal{X}_0) \cong H^3(\tilde{Y}) \oplus \on{cok}(\delta_2) \cong H^3({Y}) \oplus \on{cok}(\delta_2)$, where the last equality follows from \eqref{e:8}. This proves (iii).

For the remaining parts of (i) and (ii), we investigate the non-trivial terms of the limiting mixed Hodge diamond for $H^n := H^n(X)$:
\begin{equation} \label{H^3(Y)}
\xymatrix{&&H^{2, 2}_\infty H^3 \ar[dd]^N_\sim\\
	H^{3,0}_\infty H^3 &H^{2,1}_\infty H^3 \ar@{-}[ru]&&H^{1,2}_\infty H^3 \ar@{-}[ld]&H^{0,3}_\infty H^3, \\
	&&H^{1, 1}_\infty H^3}  
\end{equation}
where $H^{p, q}_\infty H^n = F^p_\infty \on{Gr}^W_{p + q} H^n$. The space $H^{3, 0}(X)$ does not degenerate by \cite{cW1} (which holds for degenerations with canonical singularities, and first proved in \cite{cW0} for the Calabi--Yau case). We conclude that $H^{1,1}_{\infty} H^3 \cong \on{cok}(\delta_2)$ and $\on{Gr}^W_3 H^3(X) \cong H^3(Y)$. By definition $\mu = \frac{1}{2}(h^3(X) - h^3(Y))$, hence $\mu = h^{2, 2}_\infty H^3 = h^{1, 1}_\infty H^3 = \dim \on{cok}(\delta_2)$.
\end{proof}

\subsubsection{}

We denote the \emph{vanishing cycle space} $V$ as the 
$\mathbb{Q}$-vector space generated by vanishing 3-cycles. 
We first define the abelian group $V_{\mathbb{Z}}$ from
\begin{equation} \label{e:9}
0 \to V_{\mathbb{Z}} \to H_3(X, \mathbb{Z}) \to H_3(\bar X, \mathbb{Z}) \to 0, 
\end{equation}
and $V := V_{\mathbb{Z}} \otimes_{\mathbb{Z}} \mathbb{Q}$. The sequence \eqref{e:9} arises from the homology Mayer--Vietoris sequence and the surjectivity on the right hand side follows from the fact that $H_2(\coprod^k S^3, \mathbb{Z}) = 0$. 

\begin{lemma} \label{l:7}
Denote by $H^3 := H^3(X)$.
\begin{itemize}
\item[(i)] $H^3( \bar{X} ) \cong K \cong H^3(\mathcal{X}_0) \cong W_3\, H^3$.

\item[(ii)] $V^* \cong H^{2, 2}_\infty H^3$ and $V \cong H^{1, 1}_\infty H^3 = \on{cok} (\delta_2)$ via Poincar\'e pairing.
\end{itemize}
\end{lemma}

\begin{proof}
Dualizing \eqref{e:9} over $\mathbb{Q}$, we have 
$$0 \to H^3(\bar X) \to H^3(X) \to V^* \to 0.$$ 
The invariant cycle theorem in \cite{BBD} then implies that $H^3(\bar X) \cong \ker N = K \cong H^3(\mathcal{X}_0)$. This proves (i).

Hence we have the canonical isomorphism 
$$
V^* \cong H^3(X)/H^3(\bar X) = G^W_4 H^3 = F^2_\infty G^W_4 H^3 = H^{2, 2}_\infty H^3.
$$ 
Moreover, the non-degeneracy of the pairing $(\alpha, N\beta)$ on $G_4^W H^3$ implies 
$$
H^{1, 1}_\infty H^3 = N H^{2, 2}_\infty H^3 \cong (H^{2, 2}_\infty H^3)^* \cong V^{**}_{\mathbb{C}} \cong V_{\mathbb{C}}.
$$
This proves (ii).
\end{proof}

\begin{remark}[On threefold extremal transitions] \label{r:can}
Most results in \S \ref{s:2.2} works for more general geometric contexts. The mixed Hodge diamond \eqref{H^3(Y)} holds for any 3-folds degenerations with at most canonical singularities \cite{cW1}. The identification of vanishing cycle space $V$ via \eqref{e:9} works for 3--folds with only isolated (hypersurface) singularities. Indeed, the exactness on the RHS holds for degenerations $\frak{X} \to \Delta$ such that $\frak{X}$ is smooth and $\frak{X}_0$ has only isolated singularities. This follows from Milnor's theorem that the vanishing cycle has the homotopy type of a bouquet of middle dimensional spheres \cite[Theorem 6.5]{jM}. Hence Lemma \ref{l:7} works for any 3-fold degenerations with isolated hypersurface canonical singularities. 

Later on we will impose the Calabi--Yau condition on all the 3-folds involved. If $X \nearrow Y$ is a terminal transition of Calabi--Yau 3-folds, i.e., $\mathfrak{X}_0 = \bar X$ has at most (isolated Gorenstein) terminal singularities, then $\bar X$ has unobstructed deformations \cite{N}. Moreover, the small resolution $Y \to \bar X$ induces an embedding ${\rm Def}(Y) \hookrightarrow {\rm Def}(\bar X)$ which identifies the limiting/ordinary pure Hodge structures $\on{Gr}^W_3 H^3(X) \cong H^3(Y)$ as in Corollary \ref{c:1} (iii).

For conifold transitions all these can be described in explicit terms and more precise structure will be formulated.  
\end{remark}

\subsection{The basic exact sequence} \label{s:2.3}

We may combine the four Clemens--Schmid exact sequences into one 
short exact sequence, which we call the \emph{basic exact sequence}, to give
the Hodge-theoretic realization ``$\rho + \mu = k$''
in Lemma~\ref{l:3}. 

Let $A = (a_{ij}) \in M_{k \times \mu}(\mathbb{Z})$ be a relation matrix for $C_i$'s, i.e.,
\[
 \sum\nolimits_{i=1}^k a_{ij} [C_i] = 0, \qquad j = 1, \ldots, \mu,
\]
give all relations of the curves classes $[C_i]$'s.
Similarly, let $B = (b_{ij}) \in M_{k \times \rho}(\mathbb{Z})$ be a relation matrix for $S_i$'s:
\[
 \sum\nolimits_{i=1}^k b_{ij} [S_i] = 0, \qquad j = 1, \ldots, \rho. 
\]

\begin{theorem} [Basic exact sequence] \label{t:bes}
The group of 2-cycles generated by exceptional curves $C_i$ 
(vanishing $S^2$ cycles) on $Y$ 
and the group of 3-cycles generated by $[S_i]$ (vanishing $S^3$ cycles) on $X$ 
are linked by the following weight 2 exact sequence
$$
 0 \to H^2(Y)/H^2(X) \mathop{\longrightarrow}\limits^B 
 \bigoplus\nolimits_{i = 1}^k H^2(E_i)/H^2(Q_i) 
 \mathop{\longrightarrow}\limits^{A^t}  V
  \to 0. 
$$
In particular $B = \ker A^t$ and $A = \ker B^t$.
\end{theorem}

\begin{proof}
From \S\ref{s:2.2.2}, $\on{cok}(\delta_2) = \on{cok}(\bar{\delta}_2)$
and \eqref{e:7} can be replaced by
\begin{equation} \label{e:10}
 0 \to H^2(\tilde{Y})/(\ker \bar{\delta})
  \stackrel{D}{\longrightarrow} \bigoplus\nolimits_{i = 1}^k H^2(E_i)/H^2(Q_i)
  \stackrel{C}{\longrightarrow} \on{cok}(\delta_2) \to 0.
\end{equation}
By Lemma~\ref{l:7} (ii), we have $\on{cok}(\delta_2) \cong V$.
To prove the theorem, we need to show that $H^2(\tilde Y)/\ker \bar\delta \cong H^2(Y)/H^2(X)$, and $D=B$, $C=A^t$.

By the invariant cycle theorem \cite{BBD}, $H^2(X) \cong H^2(\bar{X})$.
Since $H^2 (\bar{X})$ injects to $H^2(Y)$ by pullback, this defines the embedding 
$$\iota: H^2(X) \hookrightarrow H^2(Y)$$ 
and the quotient $H^2(Y)/H^2(X)$.

Recast the relation matrix $A$ of the rational curves $C_i$ 
in
$$
0 \to \mathbb{Q}^\mu \mathop{\longrightarrow}\limits^A \mathbb{Q}^k \cong 
\bigoplus\nolimits_{i = 1}^k \langle [C_i] \rangle \mathop{\longrightarrow}\limits^S \on{im} (\vp) \to 0
$$
where $S = \on{cok}(A) \in M_{\rho \times k}$ is the matrix for 
$\vp$, and $\on{im}(\vp)$ has rank $\rho$. The dual sequence reads
\begin{equation} \label{e:11}
 0 \to (\on{im}\vp)^* \cong (\mathbb{Q}^\rho)^* 
 \mathop{\longrightarrow}\limits^{S^t} 
 (\mathbb{Q}^k)^* \cong  \bigoplus\nolimits_{i = 1}^k H^2(E_i)/H^2(Q_i)
 \mathop{\longrightarrow}\limits^{A^t} (\mathbb{Q}^\mu)^* \to 0.
\end{equation}
Compare \eqref{e:11} with \eqref{e:10}, we see that 
$(\mathbb{Q}^\mu)^* \cong V$.
From the discussion in \S\ref{s:2.2.2}, we have 
$(\on{im} \vp)^* = H^2(Y)/H^2(X)$.

We want to reinterpret the map $A^t: (\mathbb{Q}^k)^* \to V$ in \eqref{e:11}.
This is a presentation of $V$ by $k$ generators, denoted by $\sigma_i$,
and the relation matrix of which is given by $S^t$.
If we show that $\sigma_i$ can be identified with $S_i$, 
then $(\mathbb{Q}^\mu)^* \cong V$ and $B = S^t = \ker A^t$ is the relation matrix for $S_i$'s.

Consider the following topological construction. 
For any non-trivial integral relation $\sum_{i = 1}^k a_i [C_i] = 0$, 
there is a 3-chain $\theta$ in $Y$ with $\p \theta = \sum\nolimits_{i = 1}^k a_i C_i$. Under $\psi: Y \to \bar X$, $C_i$ collapses to the node $p_i$. 
Hence it creates a $3$-cycle 
$\bar \theta := \psi_* \theta \in H_3(\bar X, \mathbb{Z}),$
which deforms (lifts) to $\gamma \in H_3(X, \mathbb{Z})$ in nearby fibers by the surjectivity in \eqref{e:9}. 
Using the intersection pairing on $H_3(X, \mathbb{Z})$, 
$\gamma$ then defines an element $\on{PD}(\gamma)$ in $H^3(X, \mathbb{Z})$. 
Under the restriction $V$, 
we get $\on{PD}(\gamma) \in V^* $. 

It remains to show that $(\gamma.S_i) = a_i$. Let $U_i$ be a small tubular neighborhood of $S_i$ and $\tilde U_i$ be the corresponding tubular neighborhood of $C_i$, then by Corollary \ref{c:surgery}, 
$$\p U_i \cong \p (S_i^3 \times D^3) \cong S^3 \times S^2 \cong \p(D^4 \times C_i) \cong \p \tilde U_i.$$ 
Now $\theta_i := \theta \cap \tilde U_i$ gives a homotopy between $a_i[C_i]$ (in the center of $\tilde U_i$) and $a_i\,{\rm pt}\times [S^2]$ (on $\p \tilde U_i$). Denote by $\iota: \p U_i \hookrightarrow X$ and $\tilde \iota: \p \tilde U_i \hookrightarrow Y$. Then
\begin{equation*}
\begin{split}
(\gamma.S_i)^X &= (\gamma.\iota_*[S^3])^X = (\iota^*\gamma.[S^3])^{\p U_i} = (\tilde\iota^*\gamma.[S^3])^{\p \tilde U_i} \\
& = (a_i[S^2], [S^3])^{S^3 \times S^2} = a_i.
\end{split}
\end{equation*}
The proof is complete.
\end{proof}

\begin{remark} \label{r:convention}
We would like to choose a preferred basis of the vanishing cocycles $V^*$ as well as a basis of divisors dual to the space of extremal curves. These notations will fixed henceforth and will be used in later sections.
 
During the proof of Theorem~\ref{t:bes}, we establish the correspondence between
$A^j = (a_{1j}, \ldots, a_{kj})^t$ and $\on{PD}(\gamma_j) \in V^*$, $1 \le j \le \mu$, characterized by $
a_{ij} = (\gamma_j.S_i)$. The subspace of $H_3(X)$ spanned by $\gamma_j$'s is denoted by $V'$.

Dually, we denote by $T_1, \ldots, T_\rho \in H^2(Y)$ those divisors which form an integral basis of the lattice in $H^2(Y)$ dual (orthogonal) to $H_2(X) \subset H_2(Y)$. In particular they form an integral basis of $H^2(Y)/H^2(X)$. We choose $T_l$'s such that $T_l$ corresponds to the $l$-th column vector of the matrix $B$ via $b_{il} = (C_i.T_l)$. Such a choice is consistent with the basic exact sequence since
$$
(A^t B)_{jl} = \sum\nolimits_{i = 1}^k a^t_{ji} b_{il} = \sum\nolimits_{i = 1}^k a_{ij} (C_i.T_l) = \Big(\sum a_{ij} [C_i]\Big).T_l = 0
$$
for all $j, l$. We may also assume that the first $\rho \times \rho$ minor of $B$ has full rank.
\end{remark}

\section{Gromov--Witten theory and Dubrovin connections} \label{s:3}

In \S\ref{s:3.1} the $A$ model $A(X)$ is shown to be a sub-theory of $A(Y)$.
We then move on to study the genus $0$ excess $A$ model on $Y/X$ associated to the extremal curve classes in \S \ref{s:3.2}.
As a consequence the (nilpotent) monodromy is calculated in terms of the relation matrix $B$ at the end of \S \ref{s:dubrovin}.

\subsection{Consequences of the degeneration formula for threefolds} \label{s:3.1}
The Gromov--Witten theory on $X$ can be related to that on $Y$ 
by the degeneration formula through
the two semistable degenerations introduced in \S\ref{s:2.1}.

In the previous section, 
we see that the monodromy acts trivially on $H(X) \setminus H^3(X)$ and we have 
$$H^3_{inv}(X) = K \cong H^3(Y) \oplus H^{1, 1}_\infty H^3(X) \cong H^3(Y) \oplus V.$$
There we implicitly have a linear map 
\begin{equation} \label{e:iota}
 \vi: H^j_{inv}(X) \to H^j (Y) 
\end{equation}
as follows.
For $j=3$, it is the projection 
$$H^3_{inv}(X) \cong H^3(Y) \oplus V \to H^3 (Y).$$ 
For $j=2$, it is the embedding defined before
and the case $j=4$ is the same as (dual to) the $j=2$ case.
For $j=0, 1, 5, 6$, $\vi$ is an isomorphism.

The following is a refinement of a result of Li--Ruan \cite{LR}. 
(See also \cite{LY}.)

\begin{proposition} \label{p:1}
Let $X \nearrow Y$ be a projective conifold transition. Given $\vec{a} \in  (H^{\ge 2}_{inv}(X)/V)^{\oplus n}$ and a curve class $\beta \in NE(X) \setminus \{0\}$, we have
\begin{equation} \label{e:deg}
 \langle \vec a \rangle_{g, n , \beta}^X = \sum\nolimits_{{\psi}_*(\gamma) = \beta} \langle \vi(\vec a)\rangle_{g, n , \gamma}^Y .
\end{equation}
If some component of $\vec{a}$ lies in $H^0$, then both sides vanish.
Furthermore, the RHS is a \emph{finite} sum. 
\end{proposition}

\begin{proof}
A slightly weaker version of \eqref{e:deg} has been proved in \cite{LR, LY}.
We review its proof with slight refinements as it will be useful in \S\ref{s:5}.

We follow the setup and argument in \cite[\S 4]{LLW1} closely. By \cite[\S 4.2]{LLW1}, a cohomology class $a \in H^{> 2}_{inv}(X)/V$ can always find 
a lift to 
\[
 (a_i)_{i=0}^k \in H(\tilde{Y}) \oplus \bigoplus\nolimits_{i=1}^k H(Q_i)
\]
such that $a_i= 0$ for all $i \neq 0$. We apply J.~Li's algebraic version of degeneration formula \cite{JL2, LY} to the complex degeneration $X \rightsquigarrow \tilde{Y} \cup_E Q$, where 
$$Q := \coprod\nolimits_{i = 1}^k Q_i $$ is a disjoint union of quadrics $Q_i$'s and 
$$E := \sum\nolimits_{i = 1}^k E_i .$$ 
One has $K_{\tilde Y} = \tilde{\psi}^* K_{\bar{X}} + E$.
The topological data $(g, n, \beta)$ lifts to two admissible triples 
$\Gamma_1$ on $(\tilde Y, E)$ and $\Gamma_2$ on $(Q, E)$ 
such that $\Gamma_1$ has curve class $\tilde{\gamma} \in NE(\tilde Y)$, 
contact order $\mu = (\tilde \gamma.E)$, and number of contact points $\rho$. 
Then 
$$(\tilde{\gamma}. c_1(\tilde Y)) = (\tilde{\psi}_*\tilde{\gamma}.c_1(\bar{X})) - (\tilde \gamma. E) = (\beta.c_1(X)) - \mu.$$ 
The virtual dimension (without marked points) is given by
\begin{equation*}
\begin{split}
d_{\Gamma_1} = (\tilde \gamma.c_1(\tilde Y)) + (\dim X - 3)(1 - g) + \rho - \mu = d_\beta + \rho - 2\mu
\end{split}
\end{equation*}
where $d_\beta$ is the virtual dimension of the absolute invariant with curve class $\beta$ (without marked points). Since we chose the lifting $(\vec a_i)_{i = 0}^k$ of $\vec a$ to have $\vec a_i = 0$ 
for all $i \ne 0$, all insertions contribute to $\tilde Y$. 
If $\rho \ne 0$ then $\rho - 2\mu < 0$. 
This leads to vanishing relative GW invariant on $(\tilde Y, E)$. 
Therefore, $\rho$ must be zero.

To summarize,  we get
\begin{equation} \label{e:i}
\langle \vec{a} \rangle^X_{g, n, \beta} = 
\sum\nolimits_{\tilde\psi_*(\tilde\gamma) = \beta} 
\langle {\vec{a}_0} \mid \emptyset\rangle^{(\tilde{Y}, E)}_{g, n, \tilde\gamma},
\end{equation}
such that
\begin{equation} \label{e:lifting}
\tilde{\psi}_* \tilde{\gamma} = \beta, \qquad \tilde{\gamma} . E =0, 
 \qquad \tilde{\gamma}_Q = 0.
\end{equation}
Formula \eqref{e:i} also holds for $a_i$ a divisor by the divisor axiom.

We use a similar argument to compute $\langle \vec b\rangle^Y_{g, n, \gamma}$ 
via the K\"ahler degeneration
$Y \rightsquigarrow \tilde{Y} \cup \tilde{E}$, 
where $\tilde{E} $ is a disjoint union of $\tilde{E}_i$
(cf.~\cite[Theorem~4.10]{LLW1}).
By the divisor equation we may assume that $\deg b_j \ge 3$ 
for all $j = 1, \ldots, n$. We choose the lifting $(\vec b)_{i = 0}^k$ of $\vec b$
such that $\vec b_i = 0$ for all $i \ne 0$. 
In the lifting $\gamma_1$ on $\tilde Y$ and $\gamma_2$ on 
$\pi: \tilde E = \coprod_i \tilde E_i \to \coprod_i C_i$, 
we must have $\gamma = \phi_*\gamma_1 + \pi_* \gamma_2$. 
The contact order is given by $\mu = (\gamma_1.E)$ which has the property that $\mu = 0$ if and only if $\gamma_1 = \phi^*\gamma$ (and hence $\gamma_2 = 0$). 
If $\rho \ne 0$ we get $d_{\Gamma_1} = d_\gamma + \rho - 2\mu < d_\gamma$
and the invariant is zero. This proves
\begin{equation} \label{e:ii}
\langle \vec b\rangle^Y_{g, n, \gamma} = \langle \phi^* \vec b \mid \emptyset\rangle^{(\tilde Y, E)}_{g, n, \phi^*\gamma} ,
\end{equation}
with ${\phi}_* \tilde{\gamma} =\gamma$, $\tilde{\gamma} . E =0$, $\tilde{\gamma}_{\tilde{E}} = 0$.

To combine these two degeneration formulas together, 
we notice that in the K\"ahler degeneration, 
$\tilde \gamma \in NE(\tilde Y)$ can have contact order $\mu = (\tilde \gamma.E) = 0$ if and only if $\tilde \gamma = \phi^*\gamma$ for some $\gamma \in NE(Y)$ (indeed for $\gamma = \phi_*\tilde \gamma$). 
Choose $\vec{b} = \vi (\vec{a})$ and \eqref{e:deg} follows. The vanishing statement (of $H^0$ insertion) 
follows from the fundamental class axiom.

Now we proceed to prove the finiteness of the sum. 
(This is not stated in \cite{LR}.) For $\phi: \tilde Y \to Y$ being the blow-up along $C_i$'s, 
the curve class $\gamma \in NE(Y)$ contributes a non-trivial invariant in the sum only if 
$\phi^* \gamma$ is effective on $\tilde Y$. By combining \eqref{e:5}, \eqref{e:i} and \eqref{e:ii}, the effectivity of $\phi^*\gamma$
forces the sum to be finite. Equivalently, the condition that $\phi^*\gamma$ is effective is equivalent to that $\gamma$ is $\T$-effective under the flop $Y \dasharrow Y'$.
(i.e.\ effective in $Y$ and in $Y'$ under the natural correspondence \cite{LLW1}). Recall that under the flop the flopping curve class in $Y$ is mapped to the negative flopping curve in $Y\rq{}$. Therefore, the sum is finite.
\end{proof}

\begin{remark} \label{r:4}
The phenomena \eqref{e:deg}, including finiteness of the sum, were observed in \cite{HKTY} for Calabi--Yau hypersurfaces in weighted projective spaces from the 
numerical data obtained from the corresponding $B$ model generating function via mirror symmetry.
\end{remark}

\begin{corollary} \label{c:2}
Gromov--Witten theory on even cohomology $GW^{ev}(X)$ (of all genera)
can be considered as a sub-theory of $GW^{ev}(Y)$.
In particular, the big quantum cohomology ring is functorial with respect to
$\vi: H^{ev}(X) \to H^{ev}(Y)$ in \eqref{e:iota}.
\end{corollary}

\begin{proof}
We first note that $\vi$ is an injection on $H^{ev}$.
Proposition~\ref{p:1} then implies that all GW invariants of $X$
with even classes can be recovered from invariants of $Y$.
The only exception, $H^0$, can be treated by the fundamental class axiom.
Therefore, in this sense that $GW^{ev}(X)$ is a sub-theory of $GW^{ev}(Y)$.

In genus zero, this can be rephrased as functoriality.
Observe that the degeneration formula also holds for $\beta = 0$. 
For $g = 0$, this leads to the equality of classical triple intersection $(a, b, c)^X = (\vi(a), \vi(b), \vi(c))^Y.$
Since the Poincar\'e pairing on $H^{ev}(X)$ is also preserved under $\vi$, 
we see that the classical ring structure on $H^{ev}(X)$ are naturally 
embedded in $H^{ev}(Y)$. 

To see the functoriality of the big quantum ring with respect to $\vi$, 
we note that $(\vi(a).C_i) = 0$ for any $a \in H^{ev}(X)$ and 
for any extremal curve $C_i$ in $Y$.
Furthermore, for the invariants associated to the extremal rays 
the insertions must involve only divisors by the virtual dimension count.
Hence for generating functions with \emph{at least one insertion} 
we also have
$$
\sum\nolimits_{\beta \in NE(X)} \langle \vec a \rangle^X_\beta q^\beta = \sum\nolimits_{\gamma \in NE(Y)} \langle \vi(\vec a) \rangle^Y_\gamma q^{\psi_*(\gamma)}.
$$
Note that the case of $H^0$ is not covered in Proposition~\ref{p:1}, 
but it can be treated by the fundamental class axiom as above.
\end{proof}

\begin{remark} \label{r:5}
It is clear that the argument and conclusion hold even if some insertions 
lie in $H^3_{inv}(X)/V \cong H^3(Y)$ by Proposition~\ref{p:1}. 
\end{remark}

The full GW theory is built on the full cohomology \emph{superspace} $H = H^{ev} \oplus H^{odd}$. 
However, the odd part is not as well-studied in the literature as the even one.
In some special cases the difficulty does not occur.

\begin{lemma} \label{l:8}
Let $X$ be a smooth minimal 3-fold with $H^1(X) = 0$. 
The non-trivial primary GW invariants are all supported on $H^2(X)$ and hence, by the divisor axiom, reduced to the case without insertion. 
More generally the conclusion holds for any curve class $\beta \in NE(X)$ with $c_1(X).\beta \le 0$ for any 3-fold $X$ with $H^1(X) = 0$.
\end{lemma}

\begin{proof}
For $n$-point invariants, the virtual dimension of $\Mbar_{g, n}(X, \beta)$ is 
$$\operatorname{vdim} = c_1(X).\beta + (\dim X - 3)(1 - g) + n \le n.$$ 
Since the appearance of fundamental class in the insertions leads to trivial invariants, we must have the algebraic degree $\deg a_i \ge 1$ for all insertions $a_i$, $i = 1, \ldots, n$. Hence in fact we must have $\deg a_i = 1$ for all $i$ and $c_1(X).\beta = 0$.
\end{proof}

\subsection{The even and extremal quantum cohomology} \label{s:3.2}
From now on, we restrict to genus zero theory.

Let $s = \sum_\epsilon {s^\epsilon \bar{T}_\epsilon}\in H^2(X)$ where 
$\bar{T}_\epsilon$'s form a basis of $H^2(X)$. 
Then the genus zero GW pre-potential on $H^2(X)$ is given by 
\begin{equation} \label{e:14}
F^X_0(s) = \sum_{n = 0}^\infty \sum_{\beta \in NE(X)} \langle s^n \rangle_{0, n, \beta}\, \frac{q^\beta}{n!} = \frac{s^3}{3!} + \sum_{\beta \ne 0} n^X_\beta q^\beta e^{(\beta.s)},
\end{equation}
where $n^X_\beta = \langle \rangle_{0, 0, \beta}^X$, and $q^\beta$ the (formal) Novikov variables. 

$F^X_0 (s)$ encodes the small quantum cohomology of $X$ (and the big quantum cohomology if $X$ is minimal by Lemma~\ref{l:8}), 
\emph{except in the topological term $s^3/(3!)$ where we need the full $s \in H^{ev}(X)$}. 

Similarly we have $F_0^Y(t)$ on $H^2(Y)$ where 
\begin{equation} \label{e:split}
t = s + u \in H^2(Y) = \vi( H^2(X)) \oplus \bigoplus\nolimits_{l = 1}^\rho \langle T_l \rangle.
\end{equation}
Namely we identify $s$ with $\vi(s)$ in $H^2(Y)$ and write $u = \sum_{l = 1}^\rho u^l T_l$. 
$F^Y_0$ can be analytically continued across those boundary faces of the K\"ahler cone corresponding to flopping contractions. 
In the case of conifold transitions $Y \searrow X$, this boundary face is naturally identified as the K\"ahler cone of $X$. 

The following convention of indices on $H^{ev}(Y)$ will be used:
\begin{itemize}
\item Lowercase Greek alphabets for indices from the subspace $\vi (H^{ev}(X))$;
\item lowercase Roman alphabets for indices from the subspace spanned by the divisors $T_l$'s and exceptional curves $C_i$'s;
\item uppercase Roman alphabets for variables from $H^{ev}(Y)$. 
\end{itemize}
The generating function associated to an extremal curve $C \cong \mathbb{P}^1$ can be derived from the well-known multiple cover formula 
$$
E_0^C(t) = \sum\nolimits_{d = 1}^\infty n^N_{d} q^{d[C]} e^{d(C.t)} = \sum\nolimits_{d = 1}^\infty \frac{1}{d^3} \,q^{d[C]} e^{d(C.t)}
$$
as $N_{C/Y} = \mathscr{O}_{\mathbb{P}^1}(-1)^{\oplus 2}$.
Define
$$E_0^Y(t) := \frac{1}{3!} t^3 + \sum\nolimits_{i = 1}^k E^{C_i}_0(t) = E_0^Y(u) + \frac{1}{3!}(t^3 - u^3),$$
where $E_0^{C_i}(t) = E_0^{C_i}(u)$ depends only on $u$. Then the degeneration formula is equivalent to the following restriction
$$
F_0^X(s) - \frac{s^3}{3!}= \Big(F_0^Y(s + u) - \frac{(s  + u)^3}{3!} - E_0^Y(u) + \frac{u^3}{3!} \Big)\Big|_{q^\gamma \mapsto q^{\psi_*(\gamma)}},
$$
where $q^{[C_i]}$'s are subject to the relations induced from the relations among $[C_i]$'s.
More precisely, let $A = (a_{ij})$ be the relation matrix and define
$$
 \mathbf{r}_j (q) := \prod\nolimits_{a_{ij} > 0} q^{a_{ij} [C_i]} - \prod\nolimits_{a_{ij} < 0} q^{-a_{ij} [C_i]}.
$$
Then we have
\begin{lemma}  \label{l:9}
$$
F_0^Y(s + u) = \left[ F_0^X(s) + E_0^Y(u) + \frac{1}{3!}((s + u)^3 - s^3 - u^3) \right]_{\mathbf{r}_j(q) =0, \,  1 \leq j \leq \mu}.
$$
\end{lemma}

A splitting of variables of $F_0^Y$ would imply that $QH^{ev}(Y)$ decomposes into two blocks. 
One piece is identified with $QH^{ev}(X)$, and another piece with contributions from the extremal rays. 
However, \emph{the classical cup product/topological terms spoil the complete splitting.}

The structural coefficients for $QH^{ev}(Y)$ are $C_{PQR} = \p^3_{PQR} F^Y_0$. 
We will determine them according to the partial splitting in Lemma \ref{l:9}.

For $F_0^X(s)$, the structural coefficients of quantum product are given by 
$$
C_{\epsilon \zeta \vi}(s) := \p^3_{\epsilon \zeta \vi} F^X_0(s) = (\bar T_\epsilon.\bar T_\zeta.\bar T_\vi) + \sum\nolimits_{\beta \ne 0} (\beta.\bar T_\epsilon)(\beta.\bar T_\zeta)(\beta.\bar T_\vi)\, n_\beta^X\, q^\beta e^{(\beta.s)}.
$$

Recall that $B = (b_{ip})$ with $b_{ip} = (C_i.T_p)$ is the relation matrix
for the vanishing 3-spheres.
For $E_0^Y(u)$, the triple derivatives are 
\begin{equation} \label{extr-inv}
\begin{split}
C_{lmn}(u) &:= \p^3_{lmn} E_0^Y(u) \\
&= (T_l.T_m.T_n) + \sum\nolimits_{i = 1}^k \sum\nolimits_{d = 1}^\infty (C_i.T_l) (C_i.T_m) (C_i.T_n)\, q^{d[C_i]} e^{d(C_i.u)} \\
&= (T_l.T_m.T_n) + \sum\nolimits_{i = 1}^k b_{il} b_{im} b_{in} \f(q^{[C_i]} \exp {\sum\nolimits_{p = 1}^\rho b_{ip}u^p}).
\end{split}
\end{equation}
Here $\f(q) = \sum_{d \in \mathbb{N}} q^d = \frac{q}{1 - q} = -1 + \frac{-1}{q - 1}$ is the fundamental rational function with a simple pole at $q = 1$ with residue $-1$ (cf.~\cite{LLW1}). 
We note that due to the existence of cross terms in Lemma \ref{l:9}, $C_{lmn}$'s do not satisfy the WDVV equations. 

Denote by $\bar T^\epsilon \in H^4(X)$ the dual basis of $\bar{T}_\epsilon$'s, and write $T^l$, $1 \le l \le \rho$ the dual basis of $T_l$'s. Also $\bar T_0 = T_0 = {\bf 1}$ with dual $\bar T^0 = T^0$ the point class. 
Since $H^{ev}(Y) = \vi( H^{ev}(X)) \oplus \big(\bigoplus_{l = 1}^\rho \mathbb{Q} T_l \oplus \bigoplus_{l = 1}^\rho \mathbb{Q} T^l \big)$ is an orthogonal decomposition with respect to the Poincar\'e pairing on $H(Y)$, we have four types of structural coefficients
\begin{equation*} 
C_{\epsilon \zeta}^\vi(s) = C_{\epsilon \zeta \vi}(s),  \quad C_{lm}^n(u) = C_{lmn}(u), \quad C_{\epsilon m}^n = C_{\epsilon m n}, \quad C_{mn}^\epsilon = C_{\epsilon mn},
\end{equation*}
where the last two are constants. 
If we consider the topological terms $\frac{1}{2} (s^0)^2 s^{0'} + s^0 \sum_{\epsilon} u^{l} u^{l'}$ where we relabel the indices by $u^{l'} = u_l$ and $s^{0'} = s_0$,
then a few more non-trivial constants $C_{000'} = 1$, $C_{mn'0} = \delta_{mn}$ are added.

\subsection{The Dubrovin connection and monodromy} \label{s:dubrovin}

The Dubrovin connection on $T H^{ev}(Y)$ is given by $\nabla^z = d - \frac{1}{z} \sum_{P} dt^P \otimes T_P *$. By Corollary \ref{c:2}, it restricts to the Dubrovin connection on $T H^{ev}(X)$. For the complement with basis $T_l$'s and $T^l$'s, we have
\begin{equation} \label{dubrovin}
\begin{split}
z\nabla^z_{\p_l} T^m &= -\delta_{lm} T^0, \\
z\nabla^z_{\p_l} T_m &= -\sum\nolimits_{n = 1}^\rho C_{lmn}(u) T^n - \sum\nolimits_{\epsilon} C_{lm \epsilon} \bar T^\epsilon, \\
z\nabla^z_{\p_\epsilon} T_m &= -\sum\nolimits_{n = 1}^\rho C_{\epsilon mn} T^n.
\end{split}
\end{equation}

Along $u = \sum_{l = 1}^\rho u^l T_l$ there is no convergence issue by the explicit expression \eqref{extr-inv}. Thus we drop the Novikov variables henceforth. 

From \eqref{extr-inv}, the degeneration loci $\mathfrak{D}$ consists of $k$ hyperplanes in $H^2(Y)$:
$$
D_i := \{ v_i := \sum\nolimits_{p = 1}^\rho b_{ip} u^p = 0 \}, \quad 1 \le i \le k,
$$ 
which is the K\"ahler degenerating locus at which $C_i$ shrinks to zero volume. There is a monodromy matrix corresponding to $D_i$, whose main nilpotent block $N_{(i)} = (N_{(i), mn}) \in M_{\rho \times \rho}$ is the residue matrix of the connection in \eqref{dubrovin}. The divisor $\mathfrak{D} = \bigcup_{i = 1}^k D_i$ is \emph{not} normal crossing.

\begin{lemma} \label{l:10}
In terms of $\{T_n\}$ and dual basis $\{T^n\}$, the block $N_{(i)}$ is given by
$$
N_{(i), mn} = \frac{1}{z} b_{im} b_{in}.
$$
\end{lemma}

\begin{proof}
Since $dv_i = \sum_{l = 1}^\rho b_{il}\, du^l$, we get from \eqref{dubrovin} and \eqref{extr-inv} that
$$
N_{(i), mn} = -\frac{1}{z} b_{im} b_{in} \mathop{{\rm Res}}\limits_{v_i = 0} \frac{-1}{e^{v_i} - 1} 
$$
which gives the result.
\end{proof}

\begin{corollary} \label{c:3}
In terms of $\{T_n\}$ and dual basis $\{T^n\}$, the nilpotent monodromy at $u = 0$ along $u^l \to 0$ has its main block given by $N_l = \frac{1}{z} B_l^t B_l$, where $B_l$ is obtained from $B$ by setting those $i$-th rows to $0$ if $b_{il} = 0$.
\end{corollary}

\begin{proof}
This follows from Lemma \ref{l:10}, which can also be proved directly. To determine $N_{l, mn}$ along $u^l \to 0$ at the locus $u = 0$, we compute
\begin{equation*}
\begin{split}
N_{l, mn} = -\frac{1}{z} \sum_{i = 1}^k b_{il} b_{im} b_{in}\, \mathop{{\rm Res}}\limits_{q = 1} \frac{-1}{e^{b_{il} u^l} - 1} = \frac{1}{z} \sum_{b_{il} \ne 0;\, i = 1}^k b_{im} b_{in} = \frac{1}{z} (B_l^t B_l)_{mn}.
\end{split}
\end{equation*}
This proves the result.
\end{proof}

\begin{corollary}
The Dubrovin connection on $X$ is the monodromy invariant sub-system on $Y$ at $u = 0$.
\end{corollary}

\section{Period integrals and Gauss--Manin connections} \label{s:4}

From this section and on, we assume the Calabi--Yau condition: 
$$K_X \cong \mathscr{O}_X, \qquad H^1(\mathscr{O}_X) =0.$$
Recall that the Kuranishi space $\mathcal{M}_{\bar{X}}$ is smooth.
In \S \ref{s:4.1}, we review well known deformation theory of Calabi--Yau 3-folds with ODPs to derive a local Torelli theorem for $\bar{X}$.
Identifying $\M_Y$ with equisingular deformations of $\bar X$ in $\M_{\bar{X}}$,
we show that periods of vanishing cycles serve as (analytic) coordinates of $\M_{\bar{X}}$ in the directions transversal to $\M_Y$.
To study monodromy, the Bryant--Griffiths formulation is reviewed in \S \ref{s:4.2} and the asymptotics of ($\beta$-)periods near $[\bar X]$ is computed in \S \ref{s:4.3}. 
The monodromy is determined explicitly in terms of the relation matrix $A$ (Corollary~\ref{c:monodromy}). 
The technical result (Theorem \ref{p:gnot}) is a version of nilpotent orbit theorem with non-SNC boundary, which is also needed in \S\ref{s:6}. 
Following these discussions, $B(Y)$ is shown to be a sub-theory of $B(X)$ (Corollary \ref{c:sub-sys}).

\subsection{Deformation theory} \label{s:4.1}
The main references for this subsection are \cite{yK, RT}, though we follow the latter more closely. Let $\Omega_{\bar{X}}$ be the sheaf of K\"ahler differential and $\Theta_{\bar{X}} := \mathscr{H}{om} (\Omega_{\bar{X}}, \mathscr{O}_{\bar{X}})$ be its dual. 
The deformation of $\bar{X}$ is governed by $Ext^1 (\Omega_{\bar{X}}, \mathscr{O}_{\bar{X}})$.
By local to global spectral sequence, we have
\begin{equation} \label{e:15}
 \begin{split}
 0 \to H^1 (\bar{X}, \Theta_{\bar{X}}) &\stackrel{\lambda}{\to} Ext^1 (\Omega_{\bar{X}}, \mathscr{O}_{\bar{X}}) \\
 &\to H^0 (\bar{X}, \mathscr{E}xt^1 (\Omega_{\bar{X}}, \mathscr{O}_{\bar{X}})) \stackrel{\kappa}{\to} H^2 (\bar{X}, \Theta_{\bar{X}}).
 \end{split}
\end{equation}
Since $\mathscr{E}xt^1 (\Omega_{\bar{X}}, \mathscr{O}_{\bar X})$ is supported at the ordinary double points $p_i$\rq{}s, we have $H^0 (\bar X, \mathscr{E}xt^1 (\Omega_{\bar{X}}, \mathscr{O}_{\bar{X}}) )= \bigoplus_{i=1}^k H^0 ( \mathscr{O}_{p_i})$ by a local computation.

We rephrase the deformation theory on $\bar{X}$ in terms of the log deformation on $\tilde{Y}$. Denote by $E \subset \tilde Y$ the union of the exceptional divisors of $\tilde{\psi}: \tilde{Y} \to \bar{X}$.

\begin{lemma} \label{l:11}
We have $R \tilde{\psi}_* K_{\tilde{Y}} = \tilde{\psi}_* K_{\tilde{Y}} = K_{\bar{X}}$ and hence
$H^0 (K_{\tilde{Y}}) \cong H^0 (K_{\bar{X}} ) \cong \mathbb{C}$.
\end{lemma}

\begin{proof} 
Apply the Serre duality for the projective morphism $\tilde{\psi}$ and we have
$  R \tilde{\psi}_* K_{\tilde{Y}} \cong (\tilde{\psi}_* \mathscr{O}_{\tilde{Y}} \otimes K_{\bar{X}})^{\vee}$. Since $\bar X$ is normal rational Gorenstein, we have $\tilde{\psi}_* \mathscr{O}_{\tilde{Y}} \cong \mathscr{O}_{\bar{X}}$. This proves the first equation, from which the first part of the second equation follows. The second part follows from $K_{\bar{X}} \cong \mathscr{O}_{\bar{X}}$.
\end{proof}

\begin{lemma} \label{l:12}
There is a canonical isomorphism
$$
\Omega^2_{\tilde{Y}} (\log E) \cong  K_{\tilde{Y}} \otimes \left( \Omega_{\tilde{Y}} (\log E)(-E) \right)^{\vee}.
$$
\end{lemma}

\begin{proof}
On $\tilde{Y}$, the isomorphism $\Lambda^3 \Omega_{\tilde{Y}}(\log E) \cong \Omega^3_{\tilde Y}(E)$ leads to the perfect pairing $\Omega_{\tilde{Y}} (\log E) \otimes \Omega^2_{\tilde{Y}} (\log E) \to K_{\tilde{Y}} (E)$. Since $\tilde{Y}$ is nonsingular and $E$ is a disjoint union of nonsingular divisors, 
all sheaves involved are locally free. Hence the lemma follows.
\end{proof}

\begin{lemma}[{\cite[Lemma~2.5]{RT}}] \label{l:13}
There are canonical isomorphisms 
$$
L \tilde{\psi}^* \Omega_{\bar{X}} \cong \tilde{\psi}^* \Omega_{\bar{X}} \cong \Omega_{\tilde{Y}} (\log E) (-E),
$$
where $L \tilde{\psi}^*$ is the left-derived functor of the pullback map.
\end{lemma}

The first isomorphism follows from the facts that $\bar{X}$ is a local complete intersection and an explicit two-term resolution of $\Omega_{\bar{X}}$ exists. We sketch the argument here and refer to \cite{RT} for details. Locally near a node, defined by \eqref{e:A1}, one has an exact sequence $0 \to \mathscr{O} \stackrel{2 \vec{x}}{\longrightarrow} \mathscr{O}^4 \to \Omega \to 0$.
Pulling it back to $\tilde{Y}$, we see that $\tilde{\psi}^*(2 \vec{x}) :\mathscr{O} \to \mathscr{O}^4$ is injective on $Y$ and therefore higher left-derived functors are zero.

The second isomorphism is obtained by a local calculation of the blowing-up of an ordinary double point. If $x_1$ is the local equation of the exceptional divisor $E$, explicit computation in \cite{RT} shows
that $ \tilde{\psi}^* \Omega_{\bar{X}}$ is locally generated by $dx_1$ and $x_1 d x_i$ for $i \neq 1$, which is exactly $\Omega_{\tilde{Y}} (\log E) (-E)$.

\begin{lemma}[{\cite[Proposition~2.6]{RT}}] \label{l:sheaf}
We have 
$$
R \mathscr{H}om (\Omega_{\bar{X}}, K_{\bar{X}}) \cong R \tilde{\psi}_* \Omega^2_{\tilde{Y}} (\log E).
$$ 
In particular, $Ext^1 (\Omega_{\bar{X}}, K_{\bar{X}})  \cong  H^1 ( \Omega^2_{\tilde{Y}} (\log E))$.
\end{lemma}

\begin{proof}
By Lemma~\ref{l:12}, 
$$
R \tilde{\psi}_* \Omega^2_{\tilde{Y}} (\log E) \cong  R \tilde{\psi}_* \mathscr{H}om (\Omega_{\tilde{Y}} (\log E)(-E), K_{\tilde{Y}}).
$$ 
By Lemma~\ref{l:13} and the projection formula, the RHS is isomorphic to 
$$
R \mathscr{H}om (\Omega_{\bar{X}},  R \tilde{\psi}_* K_{\tilde{Y}} )
 \cong R \mathscr{H}om (\Omega_{\bar{X}}, K_{\bar{X}})
$$
with the last isomorphism coming from
$R \tilde{\psi}_* K_{\tilde{Y}} \cong K_{\bar X}$
in Lemma~\ref{l:11}.
\end{proof}

From the general deformation theory, the first term $H^1 (\bar{X}, \Theta_{\bar{X}})$ in \eqref{e:15} parameterizes equisingular deformation of $\bar{X}$. 
Thanks to the theorem of Koll\'ar and Mori \cite{KM} that this extremal  contraction deforms in families, this term parameterizes deformations of $Y$. Therefore, the cokernel of $\lambda$ in (\ref{e:15}), or equivalently the kernel of $\kappa$, corresponds to deformation of the singularities. Since the deformation of $\bar{X}$ is unobstructed \cite{yK}, 
$\on{Def}(\bar{X})$ has the same dimension as $\on{Def}(X)$, which is $h^{2,1}(X)$.
Comparing the Hodge number $h^{2,1}$ of $X$ and $\bar{Y}$ 
(cf.~\S\ref{s:2}) we have the $\dim \ker (\kappa) = \mu$.

\begin{proposition} \label{p:2}
The sequence
$$
0 \to H^1 (\bar{X}, \Theta_{\bar{X}}) \stackrel{\lambda}{\to} Ext^1 (\Omega_{\bar{X}}, \mathscr{O}_{\bar{X}}) \to V^* \to 0
$$ 
is exact.
\end{proposition}
 
\begin{proof}
The residue exact sequence on $\tilde{Y}$ is 
$$
0 \to \Omega_{\tilde{Y}} \to \Omega_{\tilde{Y}} (\log E) \stackrel{\on{res}}{\longrightarrow} \mathscr{O}_E \to 0.
$$
Taking wedge product with $\Omega_{\tilde{Y}}$ we get 
$$0 \to \Omega^2_{\tilde{Y}} \to \Omega^2_{\tilde{Y}} (\log E) \stackrel{\on{res}}{\longrightarrow} \Omega_E \to 0.$$
Part of the cohomological long exact sequence reads
\[
 H^0(\Omega_E) \to H^1 (\Omega^2_{\tilde{Y}}) \to H^1 (\Omega^2_{\tilde{Y}} (\log E) ) \to H^1 (\Omega_E) \stackrel{\kappa}{\longrightarrow} H^2 (\Omega^2_{\tilde{Y}}) .
 \]
Since $H^1(E) =0$, the first term vanishes.
By Lemma~\ref{l:sheaf}, the third term is equal to $Ext^1 (\Omega_{\bar{X}}, \mathscr{O}_{\bar{X}})$.
Indeed, it is not hard to see that this exact sequence is equal to that in \eqref{e:15} (cf.\ {\cite[(3.2)]{RT}}).

Using similar arguments as in \S\ref{s:2.2.2}, we have
\[
  0 \to H^1 (\Omega^2_{\tilde{Y}}) \to H^1 (\Omega^2_{\tilde{Y}} (\log E) ) 
  \to \bigoplus\nolimits_{i=1}^k \langle (\ell_i - \ell_i\rq{} ) \rangle 
  \stackrel{\bar{\kappa}}{\longrightarrow} \frac{H^2 (\Omega^2_{\tilde{Y}})}{\bigoplus_{i=1}^k \langle (\ell_i + \ell_i\rq{} ) \rangle}  .
 \]
From \eqref{e:delta2bar} and Lemma~\ref{l:7} (ii) we have 
$$H^2(\tilde Y) \stackrel{\bar{\delta}_2}{\longrightarrow} \bigoplus\nolimits_{i = 1}^k \langle (\ell_i - \ell_i') \rangle \to V \to 0.$$
Now by comparing the dual of the maps $\bar{\delta}_2$ and $\bar{\kappa}$, we see that $\ker(\kappa) = \on{cok}  (\bar{\delta}_2)^* = V^*$. The proof is complete.
\end{proof}

This proposition shows that the deformation of $Y$ naturally embeds to that of $\bar{X}$, 
with the transversal direction given by the periods of the vanishing cycles. Moreover, the above discussion also leads to important consequences on the infinitesimal period relations on $\tilde{Y}$ and on $\bar{X}$.

\begin{corollary} \label{c:6}
On $\tilde Y$, the natural map
\[
  H^1( \left( \Omega_{\tilde{Y}} (\log E)(-E) \right)^{\vee})
  \otimes H^0 ( K_{\tilde{Y}}) \to H^1 (\Omega^2_{\tilde{Y}} (\log E))
\]
is an isomorphism.
\end{corollary}

\begin{proof}
This follows from Lemma~\ref{l:11} and Lemma~\ref{l:12}.
\end{proof}

\begin{corollary} \label{c:5}
On $\bar X$, the natural map 
\[
  H^1( R \mathscr{H}om (\Omega_{\bar{X}}, \mathscr{O}_{\bar{X}}) )
     \otimes H^0 ( K_{\bar{X}}) \to
  Ext^1 (\Omega_{\bar{X}}, K_{\bar{X}})
  \]
is an isomorphism. Indeed, both sides are isomorphic to $Ext^1 (\Omega_{\bar{X}}, \mathscr{O}_{\bar{X}})$. 
\end{corollary}

\begin{proof}
This is a reformulation of Corollary~\ref{c:6} via Lemma \ref{l:sheaf}.
\end{proof}

\begin{remark} \label{r:3.8}
Since $\bar{X}$ is rational Gorenstein, 
$R \mathscr{H}om (\Omega_{\bar{X}}, \mathscr{O}_{\bar{X}})$ has
cohomology only in degrees $0$ and $1$.
Indeed, $R^0 \mathscr{H}om (\Omega_{\bar{X}}, \mathscr{O}_{\bar{X}}) \cong \Theta_{\bar{X}}$ and 
$$
R^1 \mathscr{H}om (\Omega_{\bar{X}}, \mathscr{O}_{\bar{X}})
 \cong \mathscr{E}xt^1 (\Omega_{\bar{X}}, \mathscr{O}_{\bar{X}})
 \cong\bigoplus\nolimits_{i=1}^k \mathscr{O}_{p_i}.
$$ 
By a Leray spectral sequence argument, this gives \eqref{e:15} 
as well and 
$$
H^1( R \mathscr{H}om (\Omega_{\bar{X}}, \mathscr{O}_{\bar{X}}) )
\cong  Ext^1 (\Omega_{\bar{X}}, \mathscr{O}_{\bar{X}}).
$$
Interpreting Corollary~\ref{c:5} as a \emph{local Torelli} type theorem, we conclude that the differentiation of any non-zero holomorphic sections of the relative canonical bundle on any deformation parameter of $\bar{X}$ is non-vanishing.
\end{remark}

\subsection{Vanishing cycles and the Bryant--Griffiths/Yukawa cubic form} \label{s:4.2}

Recall the Gauss--Manin connection $\nabla^{GM}$ on
$$\mathscr{H}^n = R^n f_* \mathbb{C} \otimes \mathscr{O}_S \to S$$
for a smooth family $f: \mathcal{X} \to S$ is a flat connection with its flat sections being identified with the local system $R^n f_* \mathbb{C}$. 
It contains the integral flat sections $R^n f_* \mathbb{Z}$. 
Let $\{ \delta_p \in H_n (X, \mathbb{Z})/{\rm (torsions)} \}$ 
be a homology basis for a fixed reference fiber $X = \mathcal{X}_{s_0}$, 
with cohomology dual basis $\delta_p^*$'s in $H^n (X, \mathbb{Z})$. 
Then $\delta_p^*$ can be extended to (multi-valued) flat sections in $R^n f_* \mathbb{Z}$. For $\eta \in \Gamma(S, \mathscr{H}^n)$, we may rewrite it in terms of these flat frames with coefficients being the ``multi-valued'' period integrals ``$\int_{\delta_p} \eta$'' as $\eta = \sum_p  \delta_p^* \int_{\delta_p} \eta$. For any local coordinate system $(\s_j)$ in $S$, since $\nabla^{GM} \delta_p^* = 0$, we get
$$
 \nabla^{GM}_{\p/\p \s_j} \eta = 
 \sum_p \delta_p^* \frac{\p}{\p \s_j} \int_{\delta_p} \eta.
$$
Thus as far as period integrals are concerned, we may simply regard the Gauss--Manin connection as partial derivatives.

When the family contains singular fibers, by embedded resolution of singularities we may assume that the discriminant loci $\mathfrak{D} \subset S$ is a normal crossing divisor. It is well-known that the Gauss--Manin connection has at worst regular singularities along $\mathfrak{D}$ by the regularity theorem. Namely it admits an extension to the boundary with at worst logarithmic poles.

Let $X \nearrow Y$ be a projective conifold transition, and $V$ the corresonding space of vanishing cycles.
Since the vanishing spheres $S_i$ have trivial normal bundles in $X$, we see that $(S_i.S_j) = 0$ for all $i, j$, and hence $V$ is isotropic.
Define $V'$ to be the subspace dual to $V$ with respect to the intersection pairing in $H_3(X)$, then $V$ and $V'$ are coisotropic.
Furthermore, we have 
$$H_3(X) \cong H_3(Y) \oplus H_3(Y)^{\perp} \cong H_3(Y) \oplus V \oplus V',$$
from (the proof of) Theorem~\ref{t:bes} and Remark~\ref{r:convention}.
Let $\{ \gamma_j \}_{j=1}^{\mu}$ be a basis of $V'$ satisfying
$$
\on{PD}(\gamma_j)([S_i]) \equiv (\gamma_j.S_i) = a_{ij}, \qquad 1 \le j \le \mu,
$$
where $S_i$'s are the vanishing 3-spheres and $A = (a_{ij})$ is the relation matrix of the exceptional curves $C_i$'s.
Additionally, let $\{ \Gamma_j \}_{j=1}^{\mu}$ be the basis of $V$ dual to
$\{ \gamma_j \}_{j=1}^{\mu}$ via intersection pairing.
Namely $(\Gamma_j.\gamma_l) = \delta_{jl}$.

\begin{lemma} \label{l:15}
We may construct a symplectic basis of $H_3(X)$:
$$
\alpha_0, \alpha_1, \ldots, \alpha_h, \beta_0, \beta_1, \ldots, \beta_h, \qquad (\alpha_j.\beta_p) = \delta_{jp},
$$ 
where $h = h^{2, 1}(X)$, with $\alpha_j = \Gamma_j$, $1 \le j \le \mu$. 
\end{lemma}

\begin{proof}
Notice that $V \subset H_3(X, \mathbb{Z})$ is generated by $[S_i^3]$'s, and hence is totally isotropic. Let $W \supset V$ be a maximal isotropic subspace (of dimension $h + 1$). We first select $\alpha_j = \Gamma_j$ for $1 \le j \le \mu$ to form a basis of $V$.
We then extend it to $\alpha_1, \ldots, \alpha_h$, and set $\alpha_0 \equiv \alpha_{h + 1}$, to form a basis of $W$. 

To construct $\beta_l$, we start with any $\delta_l$ such that $(\alpha_p.\delta_l) = \delta_{pl}$. 
Such $\delta_l$'s exist by the non-degeneracy of the Poincar\'e pairing. 
We set $\beta_1 = \delta_1$. By induction on $l$, suppose that $\beta_1, \ldots, \beta_l$ have been constructed. We define
$$
\beta_{l + 1} = \delta_{l + 1} - \sum\nolimits_{p = 1}^{l} (\delta_{l + 1}.\beta_p) \alpha_p.
$$
Then it is clear that $(\beta_{l + 1}.\beta_p) = 0$ for $p = 1, \ldots, l$.
\end{proof}

With a choice of basis of $H_3(X)$, any
$\eta \in H^3(X, \mathbb{C}) \cong \mathbb{C}^{2(h + 1)}$ is identified with its ``coordinates'' given by the period integrals $\vec\eta = \big(\int_{\alpha_p} \eta, \int_{\beta_p} \eta\big)$. Alternatively, we denote the cohomology dual basis by $\alpha_p^*$ and $\beta_p^*$ so that $\alpha_j^*(\alpha_p) = \delta_{jp} = \beta_j^*(\beta_p)$. 
Then we may write
$$
\eta = \sum\nolimits_{p = 0}^h \alpha_p^* \int_{\alpha_p} \eta + \beta_p^* \int_{\beta_p} \eta.
$$
The symplectic basis property implies that $\alpha_p^*(\Gamma) = (\Gamma.\beta_p) $ and $\beta_p^*(\Gamma) = -(\Gamma.\alpha_p) = (\alpha_p.\Gamma)$. This leads to the following observation.

\begin{lemma} \label{abc}
For $1 \le j \le \mu$, we may modify $\gamma_j$ by vanishing cycles to get $\gamma_j = \beta_j$. In particular, $(\gamma_j.\gamma_l) = 0$ for $1 \le j, l \le \mu$ and $\alpha_j^*(S_i) = (S_i.\beta_j) = -a_{ij}$.
\end{lemma}

\begin{lemma} \label{l:basis}
For all $i = 1, \ldots, k$, $\on{PD}([S_i]) = -\sum_{j = 1}^\mu a_{ij}\,\on{PD}(\Gamma_j)$.
\end{lemma}

\begin{proof}
Comparing both sides by evaluating at $\alpha_l$'s and $\beta_l$'s for all $l$.
\end{proof}

Let $\eo$ be the non-vanishing holomorphic 3-form on the Calabi--Yau threefold. Bryant--Griffiths \cite{BG} showed that the $\alpha$-periods 
$\s_p = \int_{\alpha_p} \eo$ 
form the projective coordinates of the image of the period map inside $\mathbb{P}(H^3) \cong \mathbb{P}^{2h + 1}$ as a Legendre sub-manifold of the standard holomorphic contact structure. It follows that there is a holomorphic \emph{pre-potential} $u(\s_0, \ldots, \s_h)$, which is homogeneous of weight two, such that 
$u_j \equiv \frac{\p u}{\p \s_j}  = \int_{\beta_j} \eo$. 
In fact, 
\begin{equation} \label{pp:u}
u = \tfrac{1}{2} \sum\nolimits_{p = 0}^h \s_p u_p = \tfrac{1}{2} \sum\nolimits_{p = 0}^h \s_p \int_{\beta_p} \eo.
\end{equation}
Hence $\eo = \sum_{p = 0}^h (\s_p\,\alpha_p^* + u_p\, \beta_p^*)$.
In particular, 
$$
\p_j \eo = \alpha_j^* + \sum\nolimits_{p = 0}^h u_{jp}\, \beta_p^*, \qquad \p^2_{jl} \eo = \sum\nolimits_{p = 0}^h u_{jlp} \, \beta_p^*.
$$
By the Griffiths transversality, $\p_j \eo \in F^2$, $\p^2_{jl} \eo \in F^1$. Hence we have the \emph{Bryant--Griffiths cubic form}, which is homogeneous of weight $-1$:
$$
u_{jlm} = (\p_m \eo.\p^2_{jl} \eo) = \p_m (\eo.\p^2_{jl}\eo) - (\eo.\p^3_{jlm} \eo) = -(\eo.\p^3_{jlm} \eo).
$$ 
This is also known as \emph{Yukawa coupling} in the physics literature.

For inhomogeneous coordinates $z_i = \s_i/\s_0$, the corresponding formulae may be deduced from the homogeneous ones by noticing that $\p^I u$ is homogeneous of weight $2 - |I|$ for any multi-index $I$.

Under a suitable choice of the holomorphic frames respecting the Hodge filtration, \emph{the Bryant--Griffiths--Yukawa couplings determine the VHS as the structural coefficients of the Gauss--Manin connection}:

\begin{proposition} \label{BGY}
Let $\tau_0 = \eo \in F^3$, $\tau_j = \p_j \eo  \in F^2$, $\tau^j = \beta_j^* - (\s_j/\s_0) \beta_0^* \in F^1$ for $1 \le j \le h$, and $\tau^0 = \beta_0^* \in F^0$. 
Then for $1 \le p, j \le h$,
\begin{equation}
\begin{split}
\nabla_{\p_p} \tau_0 &= \tau_p, \\
\nabla_{\p_p} \tau_j &= \sum\nolimits_{m = 1}^h u_{pjm}\, \tau^m, \\
\nabla_{\p_p} \tau^j &= \delta_{pj}\,\tau^0, \\ 
\nabla_{\p_p} \tau^0 &= 0.
\end{split}
\end{equation}
\end{proposition}

\begin{proof}
We prove the second formula. Since $u_{pj}$ has weight 0, we have the Euler relation $\s_0\, u_{pj0} + \sum_{m = 1}^h \s_m\,u_{pjm} = 0$. Hence
\begin{equation*}
\begin{split}
\p_p \p_j \eo &= \sum\nolimits_{m = 1}^h u_{pjm} \,\beta_m^* + u_{pj0} \,\beta_0^* \\
&= \sum\nolimits_{m = 1}^h u_{pjm} \Big(\beta_m^* - \frac{\s_m}{\s_0} \beta_0^*\Big) = \sum\nolimits_{m = 1}^h u_{pjm} \,\tau^m.
\end{split}
\end{equation*}
It remains to show that $\tau^j \in F^1$. By the first Hodge--Riemann bilinear relations, namely $F^1 = (F^3)^\perp$ and $F^2 = (F^2)^\perp$ in our case, it is equivalent to showing that $\tau^j \in (F^3)^\perp$. This follows from
\begin{equation*}
\begin{split}
(\tau^j, \eo) = \Big(\beta_j^* - \frac{\s_j}{\s_0} \beta_0^*, \sum\nolimits_{p = 0}^h (\s_p \alpha_p^* + u_p \beta_p^*)\Big) = -\s_j + \frac{\s_j}{\s_0} \s_0 = 0.
\end{split}
\end{equation*}
The remaining statements are clear.
\end{proof}

\subsection{Degenerations via Picard--Lefschetz and the nilpotent orbit theorem} \label{s:4.3}

Let $\mathcal{X} \to \Delta$ be a one parameter conifold degeneration of 
threefolds with nonsingular total space $\mathcal{X}$.
Let $S_1, \ldots, S_k$ be the vanishing spheres of the degeneration.. 
The \emph{Picard--Lefschetz} formula (see e.g., \cite[\S 3.B]{eL}) asserts that the monodromy transformation 
$T: H^3(X) \to H^3(X)$ is given by
\begin{equation} \label{PLT}
T \sigma = \sigma + \sum\nolimits_{i = 1}^k \sigma([S_i])\, {\rm PD}([S_i]),
\end{equation}
where $\sigma \in H^3(X)$. It is \emph{unipotent}, with associated \emph{nilpotent} monodromy 
$$N := \log T = \sum\nolimits_{m = 1}^\infty (T - I)^m/m.$$
We have seen that $(S_i.S_j) = 0$ for all $i, j$. Therefore $T = I + N$ and $N^2 = 0$ (cf.~\S \ref{s:2}). 
The \emph{main purpose} here is to generalize these to multi-dimensional degenerations, and in particular to the local moduli $\M_{\bar X}$ near $[\bar X]$.

\subsubsection{VHS with simple normal crossing boundaries} \label{SNC}
Even though the discriminant loci for the conifold degenerations are in general not simple normal crossing (SNC) divisors,
by embedded resolution of singularity they can in principle be modified to become ones.
We will begin our discussion in this case for simplicity.

Let 
$$\mathcal{X} \to \mathbf{\Delta} :=  \Delta^{\nu} \times \Delta^{\nu'} \ni \bt := (t, s)$$ 
be a flat family of Calabi--Yau 3-folds such that $X_\bt$ is smooth for 
$$\bt \in \mathbf{\Delta}^* := (\Delta^\times)^{\nu} \times \Delta^{\nu'}.$$
Namely, the discriminant locus is a SNC divisor: 
$$\mathfrak{D} := \bigcup\nolimits_{j=1}^{\nu} Z( t_j) = \mathbf{\Delta} \setminus \mathbf{\Delta}^*.$$
Around each punctured disk $t_j \in \Delta^\times$, $1\le j \le \nu$, we assume the monodromy $T_j$ is unipotent with nilpotent $N_j$. Note that $N_j N_l = N_l N_j$ since $\pi_1(\mathbf{\Delta}^*) \cong \mathbb{Z}^{\nu}$ is abelian. 

If for any $\bt = (t, s)$ we assume that $X_{\bt}$ acquires at most canonical singularities, then $N_j F^3_\infty|_{D_j} = 0$ and $N_j^2 = 0$ for each $j$ (cf.~Remark \ref{r:can}). Different $N_j$ may define different weight filtration $W_j$ and each boundary divisor 
$Z(t_j)$ corresponds to different set of vanishing cycles. In our case, the structure turns out to be simple. For any $n_j \in \Bbb N$, $1 \le j \le \nu$, the degeneration along the curve 
$$\gamma(w) := (t (w), s(w)) = (w^{n_1}, \ldots, w^{n_{\nu}}, s_0)$$ 
has monodromy 
$$
N_\gamma = \log T_\gamma = \log \prod\nolimits_{j = 1}^\nu T_j^{n_j} = \sum\nolimits_{j = 1}^\nu n_j N_j.
$$ 
Hence $N_\gamma^2 = 0$ for any $(n_1, \ldots, n_\nu) \in \Bbb N^\nu$. That is, $N_{j} N_{l} = 0$ for all $j, l$. 

For conifold degenerations, this is clear from the Picard--Lefschetz formula \eqref{PLT}. Indeed $(S_{i_1}.S_{i_2}) = 0$ for all $i_1, i_2$ implies $N_j N_ l = 0$ for all $j, l$. 

Let $z_j = \log t_j / 2\pi \sqrt{-1} \in \mathbb{H}$ (the upper half plane), $\mathbf{z} N :=\sum\nolimits_{j = 1}^{\nu} z_j N_j$, and let $\eo$ denote (the class of) a relative Calabi--Yau 3-form over $\mathbf{\Delta}$, i.e.~a section of $F^3$. By Schmid's nilpotent orbit theorem \cite{Schmid} (cf.~\cite{cW0, cW1}), a natural choice of $\Omega$ takes the form
\begin{equation} \label{Omega}
\begin{split}
\eo(\bt) &= e^{\mathbf{z} N} {\bf a}(\bt) = e^{\mathbf{z} N} \Big(a_0(s) + \sum\nolimits_{j = 1}^{\nu} a_j (s) t_j + \cdots \Big) \\
&= {\bf a}(\bt) + \mathbf{z} N {\bf a}(\bt) \in F^3_\bt, 
\end{split}
\end{equation}
where ${\bf a}(\bt)$ is holomorphic, $N_j a_0(s) = 0$ for all $j$. 

In order to extend the theory of Bryant--Griffiths to include the boundary points of the period map, namely to include ODP degenerations in the current case, we need to answer the question if the $\alpha$-periods $
\theta_j(\bt) := \int_{\Gamma_j} \eo(\bt)$ may be used to replace the degeneration parameters $t_j$ for $1\le j \le \nu$. For this purpose we need to work on the local moduli space $\M_{\bar X}$.

\subsubsection{Extending Yukawa coupling towards non-SNC boundary} \label{s:4.4}

As in \S\ref{s:4.1}, $\bar X$ has unobstructed deformations
and $\M_{\bar X} = \on{Def} (\bar{X})$ is smooth. 
Since $\bar X$ admits a smoothing to $X$, $\dim \M_{\bar X}$ is exactly $h = h^{2, 1}(X)$. 
The discriminant loci $\mathfrak{D} \subset \M_{\bar X}$ is in general not a SNC divisor. 
Comparing with the local $A$ model picture on $Y/X$ in \S\ref{s:dubrovin}, the discriminant loci $\mathfrak{D}$ is expected to the union of $k$ hyperplanes. (We intentionally use the same notation $\mathfrak{D}$.)

Recall Friedman's result \cite{rF} on partial smoothing of ODPs. Let $A = [A^1, \ldots, A^\mu]$ be the relation matrix. For any $r \in \mathbb{C}^\mu$, the relation vector $A(r) := \sum_{l = 1}^\mu r_l A^l$ gives rise to a (germ of) partial smoothing of those ODP's $p_i \in \bar X$ with $A(r)_{i} \ne 0$. Thus for $1 \le i \le k$, the linear equation 
\begin{equation} \label{coor-w}
w_i :=  a_{i1} r_1 + \cdots + a_{i\mu} r_\mu = 0
\end{equation}
defines a hyperplane $Z(w_i)$ in $\mathbb{C}^\mu$. 

The small resolution $\psi: Y \to \bar X$ leads to an embedding $\M_Y \subset \M_{\bar X}$ of codimension $\mu$. 
As germs of analytic spaces we thus have $\M_{\bar X} \cong \Delta^{\mu} \times \M_Y \ni (r, s)$. 
Along each hyperplane $D^i := Z(w_i)_{\Delta^{\mu}} \times \M_Y$, there is a monodromy operator $T^{(i)}$ with associated nilpotent monodromy $N^{(i)} = \log T^{(i)}$. 
A degeneration from $X$ to $X_i$ with $[X_i] \in D^i$ a general point (not in any $D^{i'}$ with $i' \ne i$) contains only one vanishing cycle $[S^3_i] \mapsto p_i$. We summarize the above discussion in the following lemma.

\begin{lemma} \label{P-L}
Geometrically a point $(r,s) \in D^i$ corresponds to a partial smoothing $X_r$ 
of $\bar X$ for which the $i$-th ordinary double point $p_i$ remains singular. 
Hence, for $r$ generic, the degeneration from $X$ to $X_r$ has only one 
vanishing sphere $S^3_i$. 
Moreover, the Picard--Lefschetz formula \eqref{PLT} says that for any $\sigma \in H^3(X)$,
$$
N^{(i)} \sigma = (\sigma([S^3_i])) \on{PD}([S^3_i]).
$$
\end{lemma}

Even though the embedded resolution brings he discriminant locus to a SNC divisor, some information might be lost in this process.
Therefore we choose to analyze the period map directly by way of the following nilpotent orbit theorem. 
We call the configuration $\mathfrak{D} = \bigcup_{i = 1}^k D^i \subset \M_{\bar X}$ a \emph{central hyperplane arrangement with axis} $\M_Y$ following the usual convention.

\begin{theorem} \label{p:gnot}
Consider a degeneration of Hodge structures over $\Delta^\mu \times M$ with discriminant locus $\mathfrak{D}$ being a central hyperplane arrangement with axis $M$. 
Let $T^{(i)}$ be the monodromy around the hyperplane $Z(w_i)$ with quasi-unipotency $m_i$, $N^{(i)} := \log ((T^{(i)})^{m_i})/{m_i}$, 
and suppose that the monodromy group $\Gamma$ generated by $T^{(i)}$'s is \emph{abelian}. 
Let $\Bbb D$ denote the period domain and $\check{\Bbb D}$ its compact dual.
Then the period map $\phi: \Delta^\mu \times M \setminus \mathfrak{D} \to \Bbb D/\Gamma$ takes the following form 
$$
\phi(r, s) = \exp \left(\sum_{i = 1}^k \frac{m_i\log w_i}{2\pi \sqrt{-1}} N^{(i)}\right) \psi(r, s), 
$$
where $\psi: \Delta^\mu \times M \to \check{\Bbb D}$ is holomorphic and horizontal.
\end{theorem}

\begin{proof}
We prove the theorem by induction on $\mu \in \Bbb N$. The case $\mu = 1$ is essentially the one variable case (or SNC case) of the nilpotent orbit theorem. The remaining proof consists of a careful bookkeeping on Schmid's derivation of the multi-variable nilpotent orbit theorem from the one variable case (cf.~\cite[\S 8]{Schmid}, especially Lemma~(8.34) and Corollary~(8.35)). 

The essential statement is the holomorphic extension of 
\begin{equation} \label{hol.ext}
\psi(r, s) := \exp \left(-\sum_{i = 1}^k \frac{m_i\log w_i}{2\pi \sqrt{-1}} N^{(i)}\right) \phi(r, s) \in \check{\Bbb D}
\end{equation}
over the locus $\mathfrak{D}$. For $p \not\in \{0\} \times M$, we can find a neighborhood $U_p$ of $p$ so that the holomorphic extension to $U_p$ is achieved by induction. Notice that the commutativity of $N^{(i)}$'s is needed in order to arrange $\psi(r, s)$ into the form \eqref{hol.ext} with smaller $\mu$. Namely,
$$
\psi = \exp \left(-\sum_{w_i(p) = 0} \frac{m_i\log w_i}{2\pi \sqrt{-1}} N^{(i)}\right) \left[\exp \left(-\sum_{w_i(p) \ne 0} \frac{m_i\log w_i}{2\pi \sqrt{-1}} N^{(i)}\right)\phi \right].
$$

Let $R_{\ge 1/2} := \{\,(r, s)\mid |r| \ge \tfrac{1}{2} \,\}$. Then we have a unique holomorphic extension of $\psi$ over $R_{\ge 1/2}$. By the Hartog's extension theorem we get the holomorphic extension to the whole space $\Delta^\mu \times M$. The statement on horizontality follows from the same argument in \cite[\S 8]{Schmid}.
\end{proof}

\begin{remark} \label{r:abelian}
(i) Let $\mathfrak{D} = \bigcup_{i = 1}^k D^i \subset \Bbb C^\mu$ be a central hyperplane arrangement with axis $0$. Then $\Bbb C^\mu \setminus \mathfrak{D}$ can be realized as $(\Bbb C^\times)^k \cap L$ for $L \subset \Bbb C^k$ being a $\mu$ dimensional subspace. Since $\pi_1((\Bbb C^\times)^k) \cong \Bbb Z^k$, a hyperplane theorem argument shows that $\pi_1(\Bbb C^\mu \setminus \mathfrak{D}) \cong \Bbb Z^k$, hence abelian, if $\mu \ge 3$. 
However, for $\mu = 2$, $\pi_1(\Bbb C^2 \setminus \mathfrak{D})$ is \emph{not abelian} if $k \ge 3$. 
Indeed, the natural $\Bbb C^\times$ fibration $\Bbb C^2 \setminus \bigcup_{i = 1}^k D^i \to \Bbb P^1 \setminus \{p_1, \ldots, p_k\}$ leads to 
$$
0 \to \pi_1(\Bbb C^\times) \cong \Bbb Z \to \pi_1(\Bbb C^2 \setminus \bigcup D^i) \to \Bbb Z^{*(k - 1)} \to 0,
$$
where the RHS is a $k - 1$ free product of $\Bbb Z$. 

(ii)  Theorem~\ref{p:gnot} is applicable to the conifold transitions since the monodromy representation is abelian and $m_i = 1$ for all $i$. This follows from the Picard--Lefschetz formula \eqref{PLT} and the fact $[S_i].[S_{i'}]=0$ for all vanishing spheres.
\end{remark}

\begin{proposition} \label{p:3}
There is a holomorphic coordinate system $(r, s) \in \Bbb C^h$ in a neighborhood of $[\bar X] \in \M_{\bar X}$ such that $s \in \Bbb C^{h - \mu}$ is a coordinate system of $\M_Y$ near $[\bar X]$ 
and $r_j = \int_{\Gamma_j} \Omega$, $1 \le j \le \mu$, are the $\alpha$-periods of the vanishing cycles. Moreover, the section $\eo(r, s)$ takes the form
$$
\eo = a_0(s) + \sum_{j = 1}^\mu \Gamma_j^* r_j + {\rm h.o.t.} - \sum_{i = 1}^k \frac{w_i \log w_i}{2\pi \sqrt{-1}} \on{PD}([S_i]). 
$$
Here {h.o.t.}~denotes terms in $V^\perp$ which are at least quadratic in $r_1, \ldots, r_\mu$, and $w_i = a_{i1} r_1 + \cdots + a_{r\mu} r_\mu = \int_{S_i} \Omega$ defines the discriminant locus $D^i$ for $1 \le i \le k$.
\end{proposition}

\begin{proof}
By Theorem \ref{p:gnot} and the fact $N^{(i_1)} N^{(i_2)} = 0$, we may write
\begin{equation} \label{e:Omega}
\begin{split}
\Omega(r, s) &= \exp \left(\sum_{i = 1}^k \frac{\log w_i}{2\pi \sqrt{-1}} N^{(i)}\right) {\bf a}(r, s) \\
&= {\bf a}(r, s) + \sum_{i = 1}^k \frac{\log w_i}{2\pi \sqrt{-1}} N^{(i)} {\bf a}(r, s) \in F^3_{(r, s)},\end{split}
\end{equation}
where ${\bf a}(r, s) = a_0(s) + \sum_{j = 1}^\mu a_j(s)\, r_j + O(r^2)$ is holomorphic in $r, s$. 

By Lemma \ref{P-L}, all $\alpha$ periods $\theta_l := \int_{\alpha_l} \Omega$ vanish on the logarithmic terms in \eqref{e:Omega}. In particular, $\theta_l(r, s)$'s are single-valued functions. By Corollary \ref{c:5} and Remark~\ref{r:3.8} (the local Torelli property), the $h \times h$ matrix 
$$
\big(\p_m \theta_l\big)_{l, m = 1}^h = \Big(\int_{\alpha_l} \p_m\Omega \Big)
$$ 
is invertible for small $r$. Moreover, along $r = 0$, the off-diagonal block with $1 \le l \le \mu$ (i.e.~with $\alpha_l = \Gamma_l$ being the vanishing cycles) and $\mu + 1 \le m \le h$ (i.e.~with differentiation in the $s$ direction) vanishes. Hence the first $\mu \times \mu$ block
$$
\big(\p_j \theta_l \big)_{l, j = 1}^\mu = \Big(\int_{\Gamma_l} \p_j \Omega \Big)
$$
is also invertible for small $r$. Thus, by the inverse function theorem, $\theta_1, \ldots, \theta_\mu$ and $s$ form a coordinate system near $[\bar X] \in \M_{\bar X}$.

Now we replace $r_j$ by the $\alpha$-period $\theta_j$ for $j = 1, \ldots, \mu$. 
In order for Theorem~\ref{p:gnot} to be applicable, we need to justify that the discriminant locus $D^i$ is still defined by linear equations in $r_j$'s. 
This follows from Lemma~\ref{l:basis}:
\begin{equation*}
\begin{split}
\int_{S_i} \Omega = (\Omega, {\rm PD}([S_i])) &= -\sum\nolimits_{j = 1}^\mu a_{ij}(\Omega, {\rm PD}(\Gamma_j)) \\
&= -\sum\nolimits_{j = 1}^\mu a_{ij} r_j =: -w_j.
\end{split}
\end{equation*}

Denote by ${\rm h.o.t}$ be terms in $V^\perp$ which are at least quadratic in $r_j$'s. The above choice of coordinates implies that
$$
\eo = a_0(s) + \sum_{j = 1}^\mu \Gamma_j^* r_j + {\rm h.o.t.} + \sum_{i = 1}^k \sum_{j = 1}^\mu \frac{\log w_i}{2\pi \sqrt{-1}} N^{(i)} \Gamma_j^* r_j. 
$$
Then 
$$
\sum\nolimits_{j = 1}^\mu N^{(i)} \Gamma_j^* r_j = -\sum\nolimits_{j = 1}^\mu a_{ij} r_j \on{PD}([S_i]) = -w_i \on{PD}([S_i])
$$ 
by Lemma \ref{P-L} and Lemma \ref{abc}. The proof is complete.
\end{proof}

Consequently one obtains the asymptotic forms of $\beta$-periods and Bryant--Griffiths form in terms of the above coordinate system $(r, s)$.
For $\beta$-periods
\[
u_p(r, s) = \int_{\beta_p} \eo = u_p(s) + {\rm h.o.t.} - \sum_{i = 1}^k \frac{w_i \log w_i}{2\pi \sqrt{-1}} \int_{\beta_p} \on{PD}([S_i])
\]
since $\eo(s) = a_0(s)$. Thus
\[
\begin{split}
 u_p(r, s) &= u_p(s) +\sum_{i = 1}^k \frac{w_i \log w_i}{2\pi \sqrt{-1}} a_{ip} + {\rm h.o.t.} \quad   \text{for $1 \le p \le \mu$} \\
 u_p(r, s) &= u_p(s) + {\rm h.o.t.}  \quad \text{for } p > \mu.
\end{split}
\]
The Bryant--Griffiths form is then obtained by taking two more derivatives. 
For $1 \le p, m, n \le \mu$, we get 
$$
u_{pm} =  O(r) + \sum_{i = 1}^k \frac{\log w_i + 1}{2\pi \sqrt{-1}} a_{ip} a_{im}
$$ 
and
\begin{equation} \label{Yukawa}
u_{pmn} = O(1) + \sum_{i = 1}^k \frac{1}{2\pi \sqrt{-1}} \frac{1}{w_i} a_{ip} a_{im} a_{in}.
\end{equation}

\begin{remark}
The specific logarithmic function in Proposition \ref{p:3}, which is written in terms of linear combinations of $\alpha$-periods, had appeared in the literature in examples,
such as those studied in \cite[p.89]{CGGK} where there are 16 vanishing spheres with a single relation. To our knowledge, it has not been studied in this generality.
\end{remark}

\subsubsection{Monodromy calculations} \label{s:monodromy}

As a simple consequence, we determine the monodromy $N(l)$ towards the coordinate hyperplane $Z(r_l)$ at $r = 0$. 
That is the monodromy associated to the one parameter degeneration $\gamma(r)$ 
along the $r_l$-coordinate axis ($r_l \in \Delta$ and $r_{j} = 0$ if $j \ne l$). 
Let $I_l = \{i\mid a_{il} \ne 0\}$ and let $A_l$ be the matrix from $A$ by setting the $i$-th rows with $i \not\in I_l$ to $0$.

\begin{lemma} \label{v-sph}
The sphere $S^3_i$ vanishes in $Z(r_l)$ along transversal one parameter degenerations $\gamma$ if and only if $i \in I_l$, i.e., $a_{il} \ne 0$. 
\end{lemma}

\begin{proof}
The curve $\gamma$ lies in $D^i = Z(w_i)$ if and only if $a_{il} = 0$. 
Thus for those $i \not\in I_l$, the ODP $p_i$ is always present on $X_{\gamma(r)}$ along the curve $\gamma$. 
In particular the vanishing spheres along $\gamma$ are precisely those $S_i$ with $i \in I_l$.
\end{proof}

To calculate the monodromy $N(l)$, recall that (cf.~Lemma \ref{abc})
$\Gamma_j^* \equiv \alpha_j^* =  -\on{PD}(\beta_j)$. The Picard--Lefschetz formula (Lemma \ref{P-L}) then says that 
$$
N(l) \Gamma_j^* = \sum\nolimits_{i \in I_l} (\Gamma_j^*.\on{PD}([S_i]))\on{PD}([S_i]) = -\sum\nolimits_{i \in I_l} a_{ij} \on{PD}([S_i]).
$$

\begin{corollary} \label{c:monodromy}
For $1 \le p \le \mu$, 
$$\int_{\beta_p} N(l)\Gamma_j^* = -\sum_{i \in I_l} a_{ij} (S_i.\beta_p) = \sum_{i \in I_l} a_{ij} a_{ip} = (A_l^t A_l)_{jp},$$
while for $p = 0$ or $\mu + 1 \le p \le h$ we have $\int_{\beta_p} N(l) \Gamma_j^* = 0$.
\end{corollary}

\begin{corollary} \label{c:sub-sys}
The $B(Y)$ is a sub-theory of $B(X)$ by setting $r=0$ and taking the monodromy invariant sub-system. In fact $a_0(s)$ represents the family of Calabi--Yau 3-forms $\Omega(s)$ over $\M_Y$ and the $\alpha$, $\beta$ periods along it gives the VHS on $Y$. 
\end{corollary}

\subsubsection{On topological logarithmic Gauss--Manin connection}
We study the \emph{topological} logarithmic Gauss--Manin connection associated to our conifold degenerations.  
That is, we seek a topological frame of the bundle $R^3 \pi_* \mathbb{C}$ of a local family $\pi: \mathcal{X} \to \mathcal{M}_{\bar X}$ near the Calabi--Yau conifold $[\bar X]$. 
By Lemma \ref{l:7} and the Hodge diamond \eqref{H^3(Y)}, part of the frame comes naturally from $H^3(Y)$, while the remaining part is modeled on $V^*$ and $V$. 
By the same procedure as in the proof of Proposition~\ref{p:3}, the topological frame modeled on $V^* \cong H^{2, 2}_\infty H^3$ can be chosen to be
\begin{equation} \label{frame:tau}
\begin{split}
v_j &:= \exp \left(\sum_{i = 1}^k \frac{\log w_i}{2\pi \sqrt{-1}} N^{(i)} \right) \Gamma_j^* \\
&= \Gamma_J^* + \sum_{i = 1}^k \frac{\log w_i}{2\pi \sqrt{-1}} N^{(i)} \Gamma_j^* = \Gamma_j^* - \sum_{i = 1}^k \frac{\log w_i}{2\pi \sqrt{-1}} a_{ij} \on{PD}([S_i])
\end{split}
\end{equation}
for $1 \le j \le \mu$. Notice that the correction terms lie in the lower weight piece $H^{1, 1}_\infty H^3$ and $v_j$ is independent of $s$. Moreover, $v_j$ is singular along $D^i$ if and only if $a_{ij} \ne 0$, i.e., $S_i$ vanishes in $Z(r^j)$ by Lemma \ref{v-sph}.

On $V \cong H^{1, 1}_\infty H^3$, we choose the (constant) frame by
\begin{equation} \label{frame:tau2}
v^j := \exp \left(\sum_{i = 1}^k \frac{\log w_i}{2\pi \sqrt{-1}} N^{(i)} \right) \on{PD}(\Gamma_j) = \on{PD}(\Gamma_j), \qquad 1 \le j \le \mu.
\end{equation}

From \eqref{coor-w}, \eqref{frame:tau} and Lemma \ref{l:basis}, it is easy to determine the Gauss--Manin connection on this partial frame in the special directions $\p/\p r_p$'s: 
\begin{equation} \label{GM-P}
\begin{split}
\nabla^{GM}_{\p/\p r_p} v_m &= \frac{1}{2\pi \sqrt{-1}} \sum_{i = 1}^k \frac{a_{ip}}{w_i} \Big( - a_{im} \on{PD}([S_i]) \Big) \\
&= \frac{1}{2\pi \sqrt{-1}} \sum_{i = 1}^k \sum_{n = 1}^\mu \frac{a_{ip} a_{im} a_{in}}{w_i} \,v^n.
\end{split}
\end{equation}

\begin{proposition} \label{p:constP}
Near $[\bar X] \in \mathcal{M}_{\bar X}$, $\nabla^{GM}$ is regular singular along $D^i$'s and smooth elsewhere. The connection matrix $P$ on the block $V^* \oplus V$ takes the form 
$$
P = \sum_{i = 1}^k \frac{dw_i}{w_i} \otimes P^i = \sum_{i = 1}^k \frac{dw_i}{w_i} \otimes \sum_{m, n = 1}^\mu a_{im} a_{in}\, v^n \otimes (v_m)^*
$$
where $P_i$ is a constant matrix in the topological frame $v_m$'s and $v^n$'s. 
\end{proposition}

Note that there are no higher order terms in $r_j$'s and $\nabla^{GM}$ is block-diagonalized, in contrast to results in  \eqref{Yukawa} and the discussions in \S\ref{s:6} where \emph{holomorphic frames} are considered.

\section{Local transitions between $A(Y)$ and $B(X)$} \label{s:local_transitions}

The basic exact sequence in Theorem \ref{t:bes} provides a Hodge theoretic realization of the numerical identity $\mu + \rho = k$. 

Now $H^2(Y)/H^2(X) \otimes \mathbb{C} \cong \mathbb{C}^\rho$ is naturally the parameter space of the extremal Gromov--Witten invariants of the K\"ahler degeneration $\psi: Y \to \bar X$, 
and $V^* \otimes \mathbb{C} \cong \mathbb{C}^\mu$ is naturally the parameter space of periods of vanishing cycles of the complex degeneration from $X$ to $\bar X$. 
Both of them are equipped with flat connections induced from the Dubrovin and Gauss--Manin connections respectively. 
Thus it is natural to ask if there is a \emph{$\mathcal{D}$ module lift of the basic exact sequence}.

We rewrite the basic exact sequence in the form
\begin{equation*}
\xymatrix{H^2_{\mathbb{C}}(Y)/H^2_{\mathbb{C}}(X)  \cong \mathbb{C}^\rho \ar[r]^>>>>B & \mathbb{C}^k & V^*_{\mathbb{C}} \cong \mathbb{C}^\mu \ar[l]_<<<<A}
\end{equation*}
with $A^t B = 0$. This simply means that $\mathbb{C}^k$ is an orthogonal direct sum of the two subspaces ${\rm im}(A)$ and ${\rm im}(B)$. Let $A = [A^1, \ldots, A^\mu]$, $B = [B^1, \ldots, B^\rho]$, and consider the invertible matrix $S = (s^i_{j}) := [A, B] \in M_{k \times k}(\mathbb{Z})$, namely $s^i_j = a_{ij}$ for $1 \le j \le \mu$ and $s^i_{\mu + j} = b_{ij}$ for $1 \le j \le \rho$. 

Denote the standard basis of $\mathbb{C}^k$ by $e_1, \ldots, e_k$ with dual coordinates $y_1, \ldots, y_k$. Let $e^1, \ldots, e^k$ be the dual basis on $(\mathbb{C}^k)^\vee$. We consider the standard (trivial) logarithmic connection on the bundle $\underline{\mathbb{C}}^k \oplus (\underline{\mathbb{C}}^k)^\vee$ over $\mathbb{C}^k$ defined by
\begin{equation} \label{Ck-log}
\nabla = d + \frac{1}{z} \sum_{i = 1}^k \frac{d y_i}{y_i} \otimes (e^i \otimes e_i^*),
\end{equation}
where $z$ is a parameter. It is a direct sum of $k$ copies of its one dimensional version. We will show that the principal (logarithmic) part of the Dubrovin connection over $\mathbb{C}^\rho$ (cf.~\eqref{extr-inv}) as well as the Gauss--Manin connection on $\mathbb{C}^\mu$ (cf.~\eqref{Yukawa}) are all induced from this standard logarithmic connection through the embeddings defined by $B$ and $A$ respectively. 

Recall the basis $T_1, \ldots, T_\rho$ of $\mathbb{C}^\rho$ with coordinates $u^1, \ldots, u^\rho$, and the frame $T_1, \ldots, T_\rho, T^1, \ldots, T^\rho$ on the bundle $\underline{\mathbb{C}}^\rho \oplus (\underline{\mathbb{C}}^\rho)^\vee$ over $\mathbb{C}^\rho$. Notice that $T_j$ corresponds to the column vector $B^j = S^{\mu + j}$, $1 \le j \le \rho$. Let $\hat T_j$ correspond to the column vector $A^j = S^j$ for $1 \le j \le \mu$ with dual $\hat T^j$'s. Then
$$
T_j = \sum\nolimits_{i = 1}^k b_{ij}\, e_i = \sum\nolimits_{i = 1}^k s^i_{\mu + j}\, e_i,
$$
and dually 
$$e^i = \sum\nolimits_{j = 1}^\mu s^i_j \,\hat T^j + \sum\nolimits_{j = 1}^{\rho} s^i_{\mu + j}\, T^j = \sum\nolimits_{j = 1}^\mu a_{ij}\, \hat T^j + \sum\nolimits_{j = 1}^{\rho} b_{ij}\, T^j.$$

Denote by $P$ the orthogonal projection 
$$P: \underline{\mathbb{C}}^k \oplus (\underline{\mathbb{C}}^k)^\vee \to \underline{\mathbb{C}}^\rho \oplus (\underline{\mathbb{C}}^\rho)^\vee .$$ 
Using \eqref{Ck-log} 
we compute the induced connection $\nabla^P$ near $\vec 0 \in \mathbb{C}^\rho$:
\begin{equation} \label{e:7.3}
\begin{split}
\nabla^P_{T_l} T_m &= \sum\nolimits_{i,\, i' = 1}^k b_{il} b_{i' m} \big(\nabla_{e_i} e_{i'} \big)^P \\
&= \frac{1}{z} \sum_{i = 1}^k \frac{b_{il} b_{im}}{y_i}\, (e^i)^P = \frac{1}{z} \sum_{n = 1}^\rho \sum_{i = 1}^k \frac{b_{il} b_{im} b_{in}}{y_i}\, T^n.
\end{split}
\end{equation}
We compare it with the one obtained in \eqref{extr-inv} and \eqref{dubrovin}:
$$
\nabla^z_{T_l} T_m = -\frac{1}{z} \sum_{n = 1}^\rho \left( (T_l.T_m.T_n) + \sum_{i = 1}^k b_{il} b_{im} b_{in}  \frac{q_i}{1 - q_i}\right) T^n,
$$
where 
$$q_i = \exp \sum_{p = 1}^\rho b_{ip} u^p = \exp v_i .$$ 
The principal part near $u_i = 0$, $1 \le i \le \rho$, gives 
$$\frac{1}{z} \sum_{n = 1}^\rho \sum_{i = 1}^k \frac{b_{il} b_{im} b_{in}}{v_i}\, T^n ,$$
which coincides with \eqref{e:7.3} by setting $v_i = y_i$ for $1 \le i \le \rho$.
We summarize the discussion in the following: 

\begin{theorem} \label{p:qbes}
Let $X \nearrow Y$ be a projective conifold transition through $\bar X$ with $k$ ordinary double points. Let the bundle $\underline{\mathbb{C}}^k \oplus (\underline{\mathbb{C}}^k)^\vee$ over $\mathbb{C}^k$ be equipped with the standard logarithmic connection defined in \eqref{Ck-log}. Then 
\begin{itemize}
\item[(1)] The connection induced from the embedding $B: \mathbb{C}^\rho \to \mathbb{C}^k$ defined by the relation matrix of vanishing 3 spheres for the degeneration from $X$ to $\bar X$ gives rise to the logarithmic part of the Dubrovin connection on $H^2(Y)/H^2(X)$. 

\item[(2)] The connection induced from the embedding $A: \mathbb{C}^\mu \to \mathbb{C}^k$ defined by the relation matrix of extremal rational curves for the small contraction $Y \to \bar X$ gives rise to the logarithmic part of the Gauss--Manin connection on $V^*$, where $V$ is the space of vanishing 3-cycles.
\end{itemize}
\end{theorem}

Part (1) has just been proved. The proof for (2) is similar (by setting $z = 2\pi \sqrt{-1}$ and $w_i = y_i$, cf.~\eqref{Yukawa}) and is omitted. We remark that the two subspaces $B(\mathbb{C}^\rho)$ and $A(\mathbb{C}^\mu)$ are indeed defined over $\mathbb{Q}$ and orthogonal to each other, hence $A$ and $B$ determine each other up to choice of basis.

\section{From $A(X) + B(X)$ to $A(Y) + B(Y)$} \label{s:5}

In this section we prove Theorem \ref{t:0.2} (3). The main idea is to refine the GW invariants on $X$ to respect the linking data on the vanishing cycles.
The GW theory of $Y$ can then be reconstructed from the linked GW theory of $X$.

\subsection{Overview} \label{s:5.1}

\subsubsection{$B(X) \Rightarrow B(Y)$}
This is explained in \S\ref{s:4}: 
The VHS on $Y$ is contained in the logarithmic extension of VHS on $X$ 
as the monodromy invariant sub-theory along $\M_Y \subset \M_{\bar X}$. 
This is the easy part.

\subsubsection{$A(X) + B(X)_{classical} \Rightarrow A(Y)$}
What we already know about $A(Y)$ consists of the following three pieces of data: 
\begin{itemize}
\item[(1)] $A(X)$, which is given,
\item[(2)] the extremal ray invariants on divisors $\{T_l\}_{l = 1}^\rho$ determined by the relation matrix $B$ of the vanishing 3-spheres, and
\item[(3)] the cup product on $H^2(Y)$. Since $Y$ comes from surgeries on $X$ along the vanishing spheres, this is determined classically. 
\end{itemize}
The ingredient (2) obviously does not come from $A(X)$ but can be computed explicitly. 
As discussed in \S\ref{s:3.2} for $g=0$ case, 
the extremal ray invariants of all genera can be obtained from 
invariants of $(-1, -1)$ curves by the relation matrix $A$.
Therefore, the ingredients needed for (2) is local and 
independent of the transition. The genus zero case was already discussed.
The $g=1$ invariants for $(-1,-1)$ curves was computed in \cite{BCOV} (and justified in \cite{GP}) and $g \geq 2$ invariants in \cite{FP}.

We make a quick comment on reconstruction in genus zero.
Using the notations in \eqref{e:split}, 
(1)--(3) above give the initial conditions on the two coordinates slices $u = 0$ and ``$s = \infty$'' (i.e., $\beta = 0$) respectively. 
Naively one may wish to reconstruct the genus zero GW theory on the entire cohomology from these two coordinate slices.
When $Y$ is Fano, this is often possible by WDVV.
However, WDVV gives no information for Calabi--Yau 3-folds.
This issue will be resolved by studying the notion of linking data below.

\subsection{Linking data} \label{s:5.2}
The homology and cohomology discussed in this subsection are over $\mathbb{Z}$. 
As a first step, we study the topological information about the holomorphic curves in $X \setminus \bigcup_{i=1}^k S_i$ instead of in $X$.
This can be interpreted as the linking data between the curve $C$ and the set of vanishing spheres $\bigcup_{i=1}^k S_i$.
We will see that the linking data add extra information to the curve class in $X$ and enable us to recover the missing topological information in the process of transition.

\begin{remark}
As mentioned in Remark~\ref{r:1} that the vanishing sphere $S_i$ can be chosen to be Lagrangian with respect to the prescribed K\"ahler form $\omega$ on $X$. 
When $\omega$ is Ricci flat, it is expected to have special Lagrangian (SL) representatives. A proof to this was recently announced in \cite[Corollary A.2]{HS}. Assuming this, then we have $T_{[S_i]}{\rm Def}(S_i/X) \cong H^1(S_i, \mathbb{R}) = 0$ by McLean's theorem \cite{rcM}. That is, $S_i$ is rigid in the SL category. 
Thus, given a curve $C$ in $X$ we expect that $C \cap S_i = \emptyset, \forall i$. 
Furthermore, by a simple virtual dimensional count, this is known to hold for a generic almost complex structure $J$ on $TX$ (cf.\ \cite{kF}). 
But we shall proceed without these heuristics. 
\end{remark}

The plan is to assign a \emph{linking data $L$} between $C$ and $S_i$'s so that
$L$ represents a refinement of $\beta =[C]$ in $X$ and
that $L$ uniquely determines a curve class $\gamma$ in $Y$, such that
$n^X_{\beta, L} = n^Y_{\gamma}$.
With the choices of lifting $\beta$ in $Y$ being fixed (as above), this is equivalent to
saying that $L$ will uniquely determine a curve class $d\ell \in N_1(Y/\bar X)$.
Let $B_i = D_\epsilon(N_{S_i/X})$ be the $\epsilon$ open tubular neighborhood of $S_i$ in $X$ with $\epsilon$ small enough such that $C \cap B_i = \emptyset$ for all $i$. Then $\p B_i = S_\epsilon(N_{S_i/X}) \cong S_i \times S^2_\epsilon \cong S^3 \times S^2$. 
Let $M := X \setminus \bigcup_{i=1}^k B_i$.
Then the pair $(M, \p M)$ is the common part for both $X$ and $Y$. Indeed let $B^+_i = D_\delta(N_{C_i/Y})$, then
$\p B^+_i = S_\delta (N_{C_i/Y}) \cong S^3_\delta \times C_i \cong S^3 \times S^2.$
This leads to two deformation retracts 
$$
(Y, \bigcup C_i) \sim (M, \p M) \sim (X, \bigcup S_i).
$$
Consider the sequence induced by the Poincar\'e--Lefschetz duality and excision theorem for $i: \p M \hookrightarrow M$: 

\small
\begin{equation} \label{e:5.2}
\xymatrix{& H_2(M, \p M) \ar[r]^<<<<<<\sim & H^4(M)  \\H_2(C) \ar[r]^{f_*} & H_2(M) \ar@{->>}[u]_{j_*} \ar[r]^<<<<<<{\sim} & H^4(M, \p M) \ar@{->>}[u]_{j^*} \\
& \bigoplus_{i} H_2(S^3_i \times S^2_i) \ar[r]^<<<<<\sim \ar[u]_{i_*}  & H^3(\p M) \ar[u]_{\Delta^*}\\
& H_3(M, \p M) \ar[u]_{\Delta_*} \ar[r]^<<<<<<\sim & H^3(M) \ar[u]_{i^*}.}
\end{equation}
\normalsize

From the retract $(M, \p M) \sim (Y, \bigcup C_i)$ and the excision sequence
for $(Y, \bigcup C_i)$ we find $H_3(M, \p M) \to \bigoplus H_2(C_i) \to H_2(Y) \to H_2(M, \p M) \to 0$. By comparing this with the LHS vertical sequence we conclude by the five lemma that $H_2(M) \cong H_2(Y)$. In particular, the curve class in $Y$
$$
\gamma := f_*[C] \in H_2(M) \cong H_2(Y)
$$
is well defined. 

\begin{definition} \label{d:2}
The linking data $(\beta, L)$ is defined to be $f_*([C]) = \gamma$ above.
\end{definition}

From the excision sequence $(X, \bigcup S_i)$, we have 
$$
0 \to H^3(M, \p M) \to H^3(X) \to \bigoplus H^3(S_i) \to H^4(M, \p M) \to H^4(X) \to 0,
$$ 
where the retract $(M, \p M) \sim (X, \bigcup S_i)$ is used. 
Comparing with the right vertical sequence in \eqref{e:5.2},
we find $H^4(M) \cong H^4(X)$ and $h^3(X) = h^3(M) + k - \rho = h^3(M) + \mu$. Since $h^3(X) = h^3(Y) + 2\mu$, this is equivalent to
\begin{equation} \label{e:5.3}
h^3(M) = h^3(Y) + \mu .
\end{equation}

\subsection{Linked GW on $X$ $=$ non-extremal GW on $Y$} \label{s:5.3}

\subsubsection{Analysis of the moduli of stable maps to the degenerating families} \label{s:5.3.1}
We recall results in J.~Li's study of degeneration formula \cite{JL1, JL2}:
given a projective flat family over a curve $\pi: W \to \mathbb{A}^1$
such that $\pi$ is smooth away from $0 \in B$ and the central fiber 
$W_0 = Y_1 \cup Y_2$ has only double point singularity
with $D := Y_1 \cap Y_2$ a smooth (but not necessarily connected) divisor,
Li in \cite{JL1} constructed a moduli stack $\mathfrak{M}(W, \Gamma) \to \mathbb{A}^1$ which has a perfect obstruction theory and hence a virtual fundamental class 
$[\mathfrak{M}(W, \Gamma)]^{\virt}$ in \cite{JL2}.
The following properties will be useful to us.
(The notations are slightly changed.)
\begin{enumerate}
\item[(1)] For every $0 \neq t \in \mathbb{A}^1$, one has
\begin{equation*} \label{e:L1}
  \mathfrak{M}(W, \Gamma)_t = \overline{M}(X, \beta), \qquad 
 [\mathfrak{M}(W, \Gamma)]^{\virt}_t = [\overline{M}(X, \beta)]^{\virt}
\end{equation*}
where $\overline{M}(X, \beta)$ is the corresponding moduli of (absolute) stable maps.

\item[(2)] For the central fiber, the perfect obstruction theory on $\mathfrak{M}(W, \Gamma)$
induces a perfect obstruction theory on $\mathfrak{M}(W_0, \Gamma)$ and
\begin{equation*} \label{e:L2}
  [\mathfrak{M}(W_0, \Gamma) ]^{\virt}= [\mathfrak{M}(W, \Gamma) ]^{\virt} \cap \pi^{-1} (0)
\end{equation*}
is a virtual divisor of $[\mathfrak{M}(W, \Gamma) ]^{\virt}$.

\item[(3)] $\mathfrak{M}(W_0, \Gamma)$ and its virtual class are related to the relative moduli and their
virtual classes. 
For each admissible triple (consisting of gluing data) $\epsilon$, there is a ''gluing map''
\begin{equation*} \label{e:L31}
 \Phi_{\epsilon} : \mathfrak{M}(Y_1, D; \Gamma_1) \times_{D^{\rho}} \mathfrak{M}(Y_2, D; \Gamma_2)
  \to  \mathfrak{M}(W_0, \Gamma),
\end{equation*}
inducing the relation between the virtual cycles
\begin{equation*} \label{e:L3}
  [\mathfrak{M}(W_0, \Gamma)]^{\virt} =   \sum_{\epsilon} m_{\epsilon}  {\Phi_{\epsilon}}_* \Delta^{!} 
  \left([\mathfrak{M}(Y_1, D; \Gamma_1)]^{\virt} \times [\mathfrak{M}(Y_2, D; \Gamma_2)]^{\virt} \right),
\end{equation*}
where $\Delta : D^{\rho} \to D^{\rho} \times D^{\rho}$ is the diagonal morphism and $m_{\epsilon}$ is a rational number (multiplicity divided by the degree of $\Phi_{\epsilon}$).
\end{enumerate}

\subsubsection{Decomposition of $\mathfrak{M}(W_0, \Gamma)$} \label{s:5.3.4}

We study properties of $\mathfrak{M}(W_0, \Gamma)$ and their virtual fundamental classes in the setting of \S\ref{s:3.1}. Namely we specialize the discussions in \S \ref{s:5.3.1} to the two semistable degenerations constructed in \S \ref{s:2.1}.

A comprehensive comparison of the curve classes in $X$, $Y$ and $\tilde{Y}$ is collected in the following diagram.
\small
$$
\xymatrix{
H_3 (M, \p M) \ar[r] \ar[d]^=  &H_2 ( \bigcup_i E_i) \ar[r] \ar[d]^{\bar{\phi}_*} &H_2(\tilde{Y}) \ar[r] \ar[d]^{\phi_*} &H_2 (M, \p M) \ar[r] \ar[d]^{=} &0 \ar[d]^{=} \\
H_3 (M, \p M) \ar[r] \ar[d]^{=}  &H_2 ( \bigcup_i C_i) \ar[r] \ar[d]^{\bar{\chi}_*} &H_2({Y}) \ar[r] \ar[d]^{\chi_*}  &H_2 (M, \p M) \ar[r] \ar[d]^{=} &0 \ar[d]^{=} \\
H_3 (M, \p M) \ar[r]  &0 \ar[r] &H_2(X) \ar[r] &H_2 (M, \p M) \ar[r] &0 
}
$$
\normalsize

A simple diagram chasing shows that there is a unique lifting $\tilde{\gamma} \in H_2(\tilde Y)$ of $\gamma \in H_2(Y)$ satisfying \eqref{e:lifting}.
From this and the degeneration analysis for the K\"ahler degeneration $Y \rightsquigarrow \tilde{Y} \cup_E \tilde{E}$ (now the divisor $D = E = \sum_{i = 1}^k E_i$), we have the following lemma. 

\begin{lemma} \label{l:5.2}
There is a homotopy equivalence 
$$
[\overline{M}(Y, \gamma)]^{\virt} \sim [\mathfrak{M}(\tilde{Y}, E; \tilde{\gamma})]^{\virt}.
$$ (If $\pi$ can be extended to a family over $\mathbb{P}^1$, then the two cycles are rationally equivalent.)
They define the same GW invariants.
\end{lemma}

Because of this lemma, we will sometimes \emph{abuse the notation and identify 
$[\mathfrak{M}(\tilde{Y}, E; \tilde{\gamma})]^{\virt}$ with $[\overline{M}(Y, \gamma)]^{\virt}$}.

\begin{lemma} \label{l:5.3}
In the case of complex degeneration $X \rightsquigarrow \tilde{Y} \cup_E Q$ in \S\ref{s:3.1},  images of $\Phi_{\tilde{\gamma}}$ for different $\tilde{\gamma}$ 
are disjoint from each other.
\end{lemma}

\begin{proof}
This follows from Li's study on the related moduli stacks.
In this special case of $\rho=0$,
for any element in $\mathfrak{M}(W_0, \Gamma)$
there is only one way to split it into two "relative maps" (with one of them being empty). 
We note that this is not true in general, when there are more than one way of splitting
of the maps to the central fiber.
\end{proof}

Given $\beta \ne 0$, let $\tilde{\gamma}$ and $\tilde{\gamma}'$ be classes appearing in \eqref{e:i}; in particular they are non-exceptional for $\tilde \psi: \tilde Y \to \bar X$. 
We have 
$$\tilde{\gamma} - \tilde{\gamma}' = \sum\nolimits_i a_i (\ell_i -\ell'_i) ,$$ 
where $\ell_i$ and $\ell'_i$ are the $\tilde \psi$ exceptional curve classes (two rulings) in $E_i$,
because $\tilde \gamma - \tilde \gamma'$ is $\tilde \psi$ exceptional and $(\tilde \gamma - \tilde \gamma').E_i = 0$.
By Proposition~\ref{p:1}, there are only finitely many nonzero $a_i$.
For each $\tilde{\gamma}$ above, there is a unique $\gamma = \psi_* \tilde{\gamma}$ in $Y$ which is non-extremal for $\psi: Y \to \bar X$ and satisfies \eqref{e:ii}.

\begin{corollary} \label{c:5.4}
Given $\beta \ne 0$ a curve class in $X$, we can associate to it sets of non-$\tilde \psi$-exceptional curve classes $\tilde{\gamma}$ and $\gamma$ discussed above. Then
\[
  [\overline{M}(X, \beta)]^{\virt} 
  \sim \sum\nolimits_{\tilde{\gamma}} [\mathfrak{M}(\tilde{Y}, E; \tilde{\gamma})]^{\virt}
  \sim \sum\nolimits_{\gamma}  [\overline{M}(Y, \gamma)]^{\virt},
\]
where $\sim$ stands for the homotopy equivalence 
and the summations are over the above sets. The conclusion holds for any projective small resolution $Y$ of $\bar X$.
\end{corollary}

\begin{proof}
This follows from \eqref{e:i}, \eqref{e:ii} and the above discussions.
\end{proof}

Recall in \S\ref{s:5.2} we have the identification
of the linking data in 
\begin{equation} \label{e:linking}
  H_2(Y^{\circ}) = H_2(Y) = H_2(X^{\circ}) = 
 H_2 (X \setminus \bigcup\nolimits_i B_i) = H_2(\bar{X} \setminus \bar{X}^{\text{sing}})
\end{equation}
where $X \setminus \bigcup_{i=1}^k S_i =: X^{\circ} \sim M \sim Y^{\circ} := Y \setminus \bigcup_{i=1}^k C_i$ and $B_i$ is a tubular neighborhood of the vanishing sphere $S_i$. Therefore, a curve class $\gamma \in H_2(Y)$ can be identified as
a "curve class" in $X^{\circ} \sim \bar{X} \setminus \bar{X}^{\text{sing}}$, with the latter a quasi-projective variety, and we can think of $\gamma$ as a curve class in $X^{\circ}$.

\begin{proposition} \label{p:4}
For $X_t$ with $t \in \mathbb{A}^1$ very small in the degenerating family $\pi : \mathcal{X} \to \mathbb{A}^1$, we have a decomposition of the virtual class $[\overline{M}(X_t, \beta)]^{\virt}$ into 
a finite disjoint union of cycles
\[
  [\overline{M}(X_t, \beta)]^{\virt} = \coprod\nolimits_{\gamma \in H_2(X^{\circ})} [\overline{M}(X_t, \gamma)]^{\virt},
\]
where $[\overline{M}(Y, \gamma)]^{\virt} \sim [\overline{M}(X_t, \gamma)]^{\virt} \in A_{\operatorname{vdim}} \left( \overline{M}(X_t, \beta) \right)$ is a cycle class corresponding to the linking data $\gamma$ of $X_t$.
\end{proposition}

\begin{proof}
By the construction of the virtual class of the family $\pi$, we know that
the virtual classes for $X_t$ and for $X_0$ are restrictions of that for $\mathcal{X}$.
Lemma \ref{l:5.3} tells us that at $t=0$, the virtual class decomposes into a disjoint union.
By semicontinuity of connected components, we conclude that the virtual classes
for $X_t$ remain disconnected with (at least) the same number of connected components
labeled by $\gamma \in H_2(X^{\circ})$.
\end{proof}

We call the numbers defined by $[\overline{M}(X_t, \gamma)]^{\virt}$
the \emph{refined GW numbers} of $X^{\circ}$ with linking data $\gamma$, or simply \emph{linked GW invariants}.

\begin{corollary} \label{c:5.6}
The refined GW numbers of $X^{\circ}$ with linking data $\gamma$
are the same as the GW invariants of $Y$ with curve class $\gamma$,
where $\gamma$ is interpreted in two ways via \eqref{e:linking}. \end{corollary}

\section{From $A(Y) + B(Y)$ to $A(X) + B(X)$} \label{s:6}

The purpose of this section is to establish part (4) of Theorem \ref{t:0.2}. 
The main idea is to refine the $B$ model on $Y$ by studying deformations and VHS ``linked'' with the exceptional curves,
i.e., on the non-compact $Y \setminus \bigcup_i C_i$. 
From this, the full VHS of $X$ is then reconstructed via Theorem \ref{p:gnot}.

\subsection{Overview} \label{s:6.1}

\subsubsection{$A(Y) \Rightarrow A(X)$}
As is explained in \S\ref{s:3}, $A(X)$ is a sub-theory of $A(Y)$.
Indeed, $A(X)$ is obtained from $A(Y)$ by setting all extremal ray invariants 
to be zero, in addition to ``reducing the linking data'' $\gamma \in NE(Y)$ 
to $\beta \in NE(X)$. 

\subsubsection{$A(Y)_{classical} + B(Y) \Rightarrow B(X)$}

We have seen that $B(Y)$ can be considered as a sub-theory of $B(X)$.
In this section, we will show that $B(Y)$, together with the knowledge of extremal curves $\bigcup_i C_I \subset Y$ determines $B(X)$.
More precisely, we will show that the ``Hodge filtration'' underlying the variation of MHS of the quasi-projective $Y^{\circ} = Y \setminus \bigcup_i C_i$ on the first jet space of $\M_Y \subset \M_{\bar{X}}$ can be lifted uniquely to the Hodge filtration underlying the degenerating VHS of $X$.
Furthermore, the information of the Gauss--Manin connection up to the first jet is sufficient to single out the VHS of $X$.

In the next subsection, we start with a statement of compatibility of MHS which is needed in our discussion.
After that we will give a proof showing the unique determination.
As in our implication of $B(X) + A(X) \Rightarrow A(Y)$ in \S\ref{s:5}, our $A(Y) + B(Y) \Rightarrow B(X)$ implication is not constructive.

\subsection{Compatibility of the mixed Hodge structures}
Recall from \S\ref{s:4.1} that $\M_{\bar{X}}$ is smooth 
and contains $\M_Y$ in a natural manner.
Set 
$$
U := Y^{\circ} = Y \setminus \bigcup\nolimits_{i = 1}^k C_i \cong \bar{X}^{\circ} = \bar{X} \setminus \bar{X}^{\on{sing}}
$$ 
where 
$$
\bar{X}^{\text{sing}} = p := \bigcup\nolimits_{i = 1}^k \{p_i\}.
$$ 
To construct the VHS with logarithmic degeneration on $\M_{\bar X}$ near $\M_Y$, we start with the following lifting property.

\begin{proposition} \label{p:lift}
There is a short exact sequence of mixed Hodge structures
\begin{equation} \label{V-MHS}
0 \to V \to H^3(X) \to H^3(U) \to 0,
\end{equation}
where $H^3(X)$ is equipped with the limiting MHS of Schmid, 
$$
V \cong H^{1, 1}_\infty H^3(X),
$$ 
and $H^3(U)$ is equipped with the canonical mixed Hodge structure of Deligne. 
In particular, $F^3 H^3(X) \cong F^3 H^3(U)$ and $F^2 H^3(X) \cong F^2 H^3(U)$.
\end{proposition}

\begin{proof}
In the topological level, the short exact sequence \eqref{V-MHS} is equivalent to the defining sequence of the vanishing cycle space \eqref{e:9}. Indeed, since $X$ is nonsingular, $H_3(X) \cong H^3(X)$ by Poincar\'e duality. Also, 
\begin{equation} \label{U-MHS}
H_3(\bar X) = H_3(\bar X, p) \cong H_3(\tilde Y, E) \cong H^3(\tilde Y \backslash E) = H^3(U)
\end{equation}
by the excision theorem and Lefschetz duality. 

Now we consider the mixed Hodge structures. 
Since $U$ is smooth quasi-projective, it is well know that the canonical 
mixed Hodge structure on $H^3(U)$ has its Hodge diamond supported on the 
upper triangular part, i.e., with weights $\ge 3$. 
Or equivalently, the MHS on $H_3(\bar X)$ has weights $\le 3$ by duality in \eqref{U-MHS}. 
The crucial point is that Lefschetz duality is compatible with mixed Hodge structures, as stated in Lemma~\ref{L-MHS} below. 
Hence the short exact sequence \eqref{V-MHS} follows from Lemma~\ref{l:7} which is essentially the invariant cycle theorem. 

Notice that $V \cong H^{1, 1}_\infty H^3(X)$ by Lemma \ref{l:7} (ii). In particular, the isomorphisms on $F^i$ for $i = 3, 2$ follows immediately by applying $F^i$ to the sequence \eqref{V-MHS}.
\end{proof}

\begin{lemma} \label{L-MHS}
Let $Y$ be an $n$ dimensional complex projective variety, $i: Z \hookrightarrow Y$ a closed subvariety with smooth complement $j: U \hookrightarrow Y$ where $U:=Y \backslash Z$. Then the Lefschetz duality $H_i(Y, Z) \cong H^{2n - i}(U)$ is compatible with the canonical mixed Hodge structures.
\end{lemma}

This is well known in mixed Hodge theory. 
For the readers' convenience we include a proof which is communicated to us by M.~de Caltaldo.

\begin{proof}
We will make use of the structural theorem of Saito on mixed Hodge modules (MHM) \cite[Theorem 0.1]{mS} which says that there is a correspondence between the derived categories of MHM and that of perverse sheaves (cf.~Axiom A in 14.1.1 of Peters and Steenbrink's book \cite{PS}).

There is a triangle in the derived category of constructible sheaves
\[
 j_! j^! \mathbb{Q}_Y \to  \mathbb{Q}_Y \to i_* i^* \mathbb{Q}_Y .
\]
This gives maps of MHS $H^i(Y,Z) \to H^i(Y) \to H^i(Z)$
with $H^i(Y,Z)=H^i (Y, j_! j^! \mathbb{Q}_Y)$. In fact, the MHS of $H^i(Y,Z)$ can be defined by the RHS from Saito's theory, since $j_! j^! \mathbb{Q}_Y$ is a complex of MHM. 

Dualizing the above setup, we have
\begin{equation} \label{e:dualizing}
    H_i (Y,Z)=H_i (Y, j_! j^! \mathbb{Q}_Y)^*,
\end{equation}
where the LHS of \eqref{e:dualizing} having MHS for the same reason as above and compatibly with taking dual as MHS. 
Furthermore, the RHS of \eqref{e:dualizing} is
$H^{-i}_c(Y,j_*j^* \omega_Y)$
by Verdier duality, where $\omega_Y$ is the Verdier dualizing complex.
Due to the compactness of $Y$ we have
\begin{equation*}
\begin{split}
    H^{-i}_c(Y,j_*j^* \omega_Y)
 &= H^{-i} (Y, j_* j^* \omega_Y)
 = H^{-i} (U, \omega_U) \\
 &=H^{BM}_i (U)
 =H^{2n - i} (U),
 \end{split}
\end{equation*}
where $H^{BM}$ is the Borel--Moore homology.
Since all steps are compatible with MHM, the Lefschetz duality is compatible with the MHS.
\end{proof}

\subsection{Conclusion of the proof} \label{s:pf}

We now apply the above result to our setting.
We have on $\bar X$ (cf.\ \cite{NS})
\[
\cdots H^1_{p} (\Theta_{\bar X}) \to H^1(\Theta_{\bar X}) \to H^1(U, T_U) \to H^2_{p} (\Theta_{\bar X}) \to \cdots.
\]
Since each $p_i$ is a hypersurface singularity, we have ${\rm depth}\,\mathscr{O}_{p_i} = 3$. Using this fact, Schlessinger \cite{Sch} (see also \cite{rF}) showed that 
$H^1_p(\Theta_{\bar X}) = 0$ and $H^2_p(\Theta_{\bar X}) \cong \bigoplus_{i = 1}^k \mathbb{C}_{p_i}$. Putting these together, we have
\begin{equation} \label{e:6.3.1}
 0 \to H^1(\Theta_{\bar X}) \to H^1(U, T_U) \to H^2_{p} (\Theta_{\bar X}) \to \cdots.
\end{equation}

Since $\bar X$ is a Calabi--Yau 3-fold with only ODPs, its deformation theory is unobstructed by the $T^1$-lifting property \cite{yK}. 
Comparing \eqref{e:6.3.1} with \eqref{e:15} we see that $\on{Def} (\bar{X}) \cong H^1(U, T_U)$.

Similarly, on $Y$ we have
$$
\cdots H^1_Z(T_Y) \to H^1(T_Y) \to H^1(U, T_U) \to H^2_Z(T_Y) \to H^2(T_Y) \to \cdots ,
$$
where $Z = Y \setminus U$ is the union of exceptional curves. Since $Y$ is smooth, the depth argument also gives $H^1_Z(T_Y) = 0$ (or by the local duality theorem $H^1_Z (T_Y) \cong H^2(Z, T_Y^\vee \otimes K_Y)^\vee = 0$). Thus 
$$
\on{Def} (Y) = H^1(T_Y) \subset H^1(U, T_U) \cong \on{Def} (\bar{X}),
$$ 
and $\M_Y$ is naturally a submanifold of $\M_{\bar{X}}$. Write $\mathscr{I} := \mathscr{I}_{\M_{Y}}$ as the ideal sheaf of $\M_Y \subset \M_{\bar{X}}$. Since $H^2(U, T_U) \neq 0$, the deformation of $U$ could be obstructed. Nevertheless, the first-order deformation of ${U}$ exists and is parameterized by $H^1(U, T_U)  \supset \on{Def} (Y)$.
Therefore, we have the following \emph{smooth family}
\[
 \pi: \mathfrak{U} \to  \Z:= Z_{\M_{\bar{X}}}(\mathscr{I}^2) \supset \M_Y,
\]
where $\Z = Z_{\M_{\bar{X}}}(\mathscr{I}^2)$ stands for the nonreduced subscheme 
of $\M_{\bar{X}}$ defined by the ideal sheaf $\mathscr{I}^2$. Namely $\Z$ is the first jet extension of $\M_Y$ in $\M_{\bar X}$.

Now we may complete the construction of VHS over $\M_{\bar X}$ near the boundary loci $\M_Y \hookrightarrow \M_{\bar X}$. The Gauss--Manin connection for a smooth family over non-reduced base was constructed in \cite{Katz}. For our smooth family $\pi: \mathcal{U} \to \Z$, it is defined by the integral lattice $H^3(U, \mathbb{Z}) \subset H^3(U, \mathbb{C})$. Since $U$ is only quasi-projective, the Gauss--Manin connection underlies VMHS instead of VHS. 
By Proposition \ref{p:lift}, we have $W_i H^3(U) = 0$ for $i \le 2$, 
$W_3 \subset W_4$ with $\on{Gr}^W_3 H^3(U) \cong H^3(Y)$, 
and $\on{Gr}^W_4 H^3(U) \cong V^*$. 

The Hodge filtration of the local system $F^0 = H^3(U, \mathbb{C})$ has the following structure: 
$F^\bullet = \{F^3 \subset F^2 \subset F^1 \subset F^0\}$ which satisfies the Griffiths transversality. 
Since $K_U \cong \mathscr{O}_U$ and $H^0(U, K_U) \cong H^0(Y, K_Y) \cong \Bbb C$, 
$F^3$ is a line bundle over $\Z$ spanned by a nowhere vanishing relative holomorphic 3-form $\eo \in \Omega^3_{\mathcal{U}/\Z}$. 
Near the moduli point $[Y] \in \Z$, $F^2$ is then spanned by $\eo$ and $v (\eo)$ where $v$ runs through a basis of $H^1(U, T_U)$. Notice that $v(\eo) \in W_3$ precisely when $v \in H^1(Y, T_Y)$. 

By Proposition~\ref{p:lift}, the partial filtration $F^3 \subset F^2$ on $H^3(U)$ over $\Z$ lifts uniquely to a filtration $\tilde F^3 \subset \tilde F^2$ on $H^3(X)$ over $\Z$ with $\tilde F^3 \cong F^3$ and $\tilde F^2 \cong F^2$. The complete lifting $\tilde F^\bullet$ is then uniquely determined since $\tilde F^1 = (\tilde F^3)^\perp$ by the first Hodge--Riemann bilinear relation on $H^3(X)$. Alternatively, $\tilde F^1$ is spanned by $\tilde F^2$ and $v (\tilde F^2)$ for $v$ runs through a basis of $H^1(U, T_U)$. 

Now $\tilde F^\bullet$ over $\Z$ uniquely determines a horizontal map 
$\Z \to \check{\Bbb D}$. Since it has maximal tangent dimension $h^1(U, T_U) = h^1(X, T_X)$, it determines uniquely the maximal horizontal slice $\psi: \M \to \check{\Bbb D}$ with $\M \cong \M_{\bar X}$ locally near $\M_Y$. The smoothing loci of $\bar X$ in $\M_{\bar X}$ is precisely given by $\M_X$.
According to Theorem \ref{p:gnot}, namely an extension of Schmid's nilpotent orbit theorem, under the coordinates ${\bf t} = (r, s)$, the period map 
$$
\phi: \M_X = \M_{\bar X} \backslash \bigcup\nolimits_{i = 1}^k D^i \to \Bbb D/\Gamma
$$ 
is then given by
$$
\phi(r, s) = \exp  \left(\sum_{i = 1}^k \frac{\log w_i}{2\pi \sqrt{-1}} N^{(i)}\right) \psi(r, s),
$$
where $\Gamma$ is the monodromy group generated by the local monodromy $T^{(i)}= \exp N^{(i)}$ (with $m_i = 1$) around the divisor $D^i$ defined by $w_i = \sum_{j = 1}^\mu a_{ij} r_j = 0$ (cf.~\eqref{coor-w}). Since $N^{(i)}$ is determined by the Picard--Lefschetz formula (Lemma \ref{P-L}), we see that the period map $\phi$ is completely determined by the relation matrix $A$ of the extremal curves $C_i$'s. (The period map gives the desired VHS, with degenerations, over $\M_X$.) This completes the proof that refined $B$ model on $Y \backslash Z = U$ determines the $B$ model on $X$.

\end{document}